\documentclass[a4paper, 11pt, fleqn, leqno]{article}
\RequirePackage{fullpage}

\usepackage{amsthm,amsmath,amssymb}
\usepackage{subcaption}
\usepackage[dvipsnames]{xcolor}
\usepackage[pdftex,breaklinks,colorlinks,
    citecolor={BlueViolet}, linkcolor={Blue},urlcolor=Maroon]{hyperref}
\usepackage{tikz}
\usetikzlibrary{decorations, decorations.markings, decorations.pathmorphing, decorations.shapes}
\usetikzlibrary{shapes}
\usetikzlibrary{arrows.meta}
\usetikzlibrary{quotes}
\usetikzlibrary{calc}
\usetikzlibrary{positioning}
\pgfdeclarelayer{bg}    \pgfsetlayers{bg,main}  \usepackage{charter,eulervm}
\usepackage{enumerate}
\usepackage[final,expansion=alltext,protrusion=true]{microtype}

\theoremstyle{plain} 
\newtheorem{theorem}{Theorem}[section]
\newtheorem{lemma}[theorem]{Lemma}

\newtheorem{observation}[theorem]{Observation}
\newtheorem{proposition}[theorem]{Proposition}
\newtheorem*{definition}{Definition}
\newtheorem{problem}{Problem}

\newcommand{\stpath}[2]{$#1$--$#2$ path}

\tikzset{
  b-vertex/.style = {fill=RubineRed, diamond, draw=blue, inner sep=1.5pt},
  a-vertex/.style = {draw=blue, fill=cyan, inner sep=2pt},
  witnessed edge/.style = {ultra thick, teal},
  region/.style = {rounded corners, draw=gray, fill=gray!40},
  filled vertex/.style = {circle,draw=black,fill=black!50,inner sep=1.5pt},
  gf-vertex/.style = {circle,draw=blue,fill=cyan,inner sep=1.2pt, label distance=-0.05cm}, gnf-vertex/.style = {circle, draw, fill = white, inner sep=1.2, label distance=-0.05cm}, empty vertex/.style = {circle, draw, fill = white, inner sep=1},
  f-vertex/.style = {fill=black!50, draw=blue, inner sep=2pt, label distance=-0.05cm},
  nf-vertex/.style = {circle, draw, fill = white, inner sep=1.2, label distance=-0.05cm},
}

\tikzset{shifted path/.style args={from #1 to #2}{insert path={
let \p1=($(#1.east)-(#1.center)$),
\p2=($(#2.east)-(#2.center)$),\p3=($(#1.center)-(#2.center)$),
\n1={.75/veclen(\x1,\y1)},\n2={.75/veclen(\x2,\y2)},\n3={atan2(\y3,\x3)} in
(#1.{\n3+180+asin(\n1)}) to (#2.{\n3-asin(\n2)})
}}}
\newcommand{\uncertain}[2]
{
  \draw[shifted path=from #1 to #2];
  \draw[densely dotted, thick, shifted path=from #2 to #1];
}

\newcommand{\strictsubset}{\sqsubset}

\newcommand{\lc}{\ensuremath{\operatorname{lc}}}
\newcommand{\rc}{\ensuremath{\operatorname{rc}}}
\newcommand{\lp}{\ensuremath{\operatorname{lp}}}
\newcommand{\rp}{\ensuremath{\operatorname{rp}}}
\newcommand{\lk}{\ensuremath{\operatorname{lk}}}
\newcommand{\rk}{\ensuremath{\operatorname{rk}}}
\newcommand{\cliques}{\ensuremath{\mathcal{K}}}
\newcommand{\inner}{\ensuremath{\operatorname{In}}}
\newcommand{\encomp}{\ensuremath{\operatorname{Ex}}}
\newcommand{\univ}{\ensuremath{\operatorname{Un}}}
\newcommand{\extreme}{\ensuremath{\operatorname{Az}}}
\newcommand{\pqtree}{T}
\newcommand{\interval}{\ensuremath{\operatorname{span}}}

\title{Characterization of Circular-arc Graphs: \\II. McConnell Flipping}
\author{
  Yixin Cao\thanks{Department of Computing, Hong Kong Polytechnic University, Hong Kong, China.  \texttt{yixin.cao@polyu.edu.hk}.
    The author gratefully acknowledges the support of the K. C. Wong Education Foundation.
    } 
  \and Tomasz Krawczyk\thanks{Faculty of Mathematics and Information Science, Warsaw University of Technology, Poland. \texttt{tomasz.krawczyk@pw.edu.pl}.}
}
\date{}

\begin{document}
\maketitle

\begin{abstract}
  McConnell [FOCS 2001] presented a flipping transformation from circular-arc graphs to interval graphs with certain patterns of representations.  Beyond its algorithmic implications, this transformation is instrumental in identifying all minimal graphs that are not circular-arc graphs.  We conduct a structural study of this transformation, and for $C_{4}$-free graphs, we achieve a complete characterization of these patterns.  This characterization allows us, among other things, to identify all minimal chordal graphs that are not circular-arc graphs in a companion paper.
\end{abstract}

\section{Introduction}\label{sec:intro}

All graphs discussed in this paper are finite and simple. 
The vertex set and edge set of a graph~$G$ are denoted by, respectively, $V(G)$ and~$E(G)$.
For a subset~$U\subseteq V(G)$, we denote by $G[U]$ the subgraph of~$G$ induced by~$U$, and by~$G - U$ the subgraph~$G[V(G)\setminus U]$, which is shortened to~$G - v$ when $U = \{v\}$.
The \emph{neighborhood} of a vertex~$v$, denoted by~$N_{G}(v)$, comprises vertices adjacent to~$v$, i.e., $N_{G}(v) = \{ u \mid uv \in E(G) \}$, and the \emph{closed neighborhood} of~$v$ is $N_{G}[v] = N_{G}(v) \cup \{ v \}$.
The \emph{closed neighborhood} and the \emph{neighborhood} of a set~$X\subseteq V(G)$ of vertices are defined as~$N_{G}[X] = \bigcup_{v \in X} N_{G}[v]$ and~$N_{G}(X) =  N_{G}[X] \setminus X$, respectively.
We may drop the subscript if the graph is clear from the context.

A graph is a \emph{circular-arc} graph if its vertices can be assigned to arcs on a circle such that two
vertices are adjacent if and only if their corresponding arcs intersect. 
Such a set of arcs is called a \emph{circular-arc model} for this graph; see Figure~\ref{fig:normal-and-helly}. 
If we replace the circle with the real line and arcs with intervals, we end with interval graphs. 
All interval graphs are circular-arc graphs. 
Both graph classes are by definition \emph{hereditary}, i.e., closed under taking induced subgraphs.

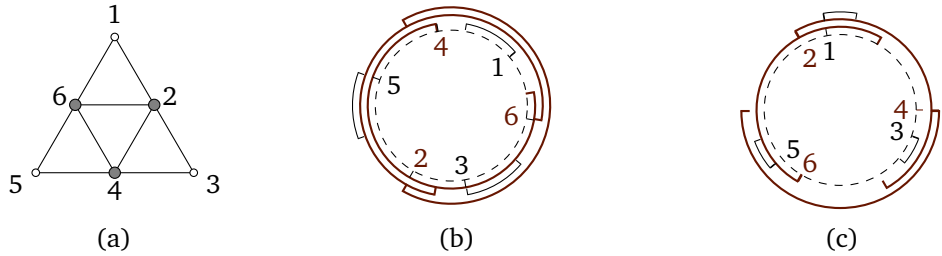
\begin{figure}[ht!]
  \centering\small
  \begin{subfigure}[b]{.23\linewidth}
    \centering
    \begin{tikzpicture}[scale=.6]\small
      \foreach \i in {1, 2, 3} 
      \draw ({120*\i-90}:1) -- ({120*\i+30}:1) -- ({120*\i-30}:2) -- ({120*\i-90}:1);
      \foreach[count =\i from 0] \n in {1, 3, 5} {
        \node[empty vertex] at ({90 - 120*\i}:2) {};
        \node at ({90 - 120*\i}:2.5) {$\n$};
        \pgfmathsetmacro{\x}{int(\n+1)}\node[filled vertex] at ({30 - 120*\i}:1) {};
        \node at ({30 - 120*\i}:1.4) {$\x$};
      }
    \end{tikzpicture}
    \caption{}
  \end{subfigure}  
  \;
  \begin{subfigure}[b]{.3\linewidth}
    \centering
    \begin{tikzpicture}[scale=.1]
      \begin{scope}[every path/.style={{|[left]}-{|[right]}}]        
        \draw[Sepia,thick] (120:13) arc (120.:-120:13); \draw (200:13) arc (200:160:13);
        \draw[Sepia,thick] (260:12) arc (260:-10:12); \draw (-40:12) arc (-40.:-80:12);
        \draw[Sepia,thick] (370:11) arc (370:100:11); \draw (80:11) arc (80:40:11);
      \end{scope}
      \draw[dashed,thin] (10,0) arc (0:360:10);
      \foreach[count=\j] \i/\c in {40/, -120/Sepia, -80/, 100/Sepia, 160/, -10/Sepia} {
        \draw[dashed] (\i:11) -- (\i:10);
        \node[\c] at (\i:8) {${\j}$};
      }
    \end{tikzpicture}
    \caption{}
  \end{subfigure}  
  \;
  \begin{subfigure}[b]{.3\linewidth}
    \centering
    \begin{tikzpicture}[scale=.1]
      \begin{scope}
        \foreach[count=\j from 0] \p/\q in {6/3, 2/5, 4/1} {
          \pgfmathsetmacro{\radius}{11+\j}
          \pgfmathsetmacro{\span}{180}
          \pgfmathsetmacro{\start}{240 - 120*\j}
          \draw[{|[left]}-{|[right]}, thick, Sepia]  (\start:\radius) arc (\start:{\start-\span}:\radius); 
          \draw[Sepia, dashed] ({\start}:\radius) -- ({\start}:10);
          \node[Sepia] at ({\start}:8) {${\p}$};

          \pgfmathsetmacro{\span}{20}
          \pgfmathsetmacro{\start}{340 - 120*\j}
          \draw[{|[left]}-{|[right]}]  (\start:\radius) arc (\start:{\start-\span}:\radius);
          \draw[dashed] ({\start}:\radius) -- ({\start}:10);
          \node at ({\start}:8) {${\q}$};
        }
      \end{scope}
      \draw[dashed,thin] (10,0) arc (0:360:10);
    \end{tikzpicture}
    \caption{}
  \end{subfigure}  
  \caption{\small A circular-arc graph and its two circular-arc models.  In (b), any two arcs for vertices~$\{2, 4, 6\}$ cover the circle; in (c), the three arcs for vertices~$\{2, 4, 6\}$ do not share any common point.}
  \label{fig:normal-and-helly}
\end{figure}

Although circular-arc graphs and interval graphs are defined in a quite similar way,
they turn out to have significantly different algorithmic and combinatorial properties.
A number of problems that are solved (or shown to admit polynomial-time solutions) 
in the class of interval graphs, are still open (are computationally hard, respectively) in the class of circular-arc graphs.
One example is the minimum coloring problem, 
which admits a simple linear algorithm 
for interval graphs, but is NP-complete on circular-arc graphs~\cite{GareyJMP80}.
Another example is the problem of determining all minimal \emph{forbidden induced subgraphs}, i.e., minimal graphs that are not in the class.
For interval graphs, the list of minimal forbidden induced subgraphs was compiled by Lekkerkerker and Boland~\cite{lekkerkerker-62-interval-graphs} in 1960's: it contains the holes---induced cycles of length at least four---and graphs shown in Figure~\ref{fig:non-interval}.
\begin{theorem}[\cite{lekkerkerker-62-interval-graphs}]
\label{thm:lb}
  A graph~$G$ is an interval graph if and only if it does not contain any hole or any graph in Figure~\ref{fig:non-interval} as an induced subgraph.
\end{theorem}

\begin{figure}[ht]
  \tikzstyle{every node}=[empty vertex]
  \centering \small
  \begin{subfigure}[b]{0.22\linewidth}
    \centering
    \begin{tikzpicture}[xscale=.6, yscale=.6]
      \node (c) at (0, 0) {};
      \foreach[count=\i] \p in {below, right, below} {
        \node (u\i) at ({90*(3-\i)}:2) {};
        \node (v\i) at ({90*(3-\i)}:1) {};
        \draw (u\i) -- (v\i) -- (c);
      }
    \end{tikzpicture}
    \caption{long claw}\label{fig:long-claw-unlabeled}
  \end{subfigure}
  \,
  \begin{subfigure}[b]{0.22\linewidth}
    \centering
    \begin{tikzpicture}[xscale=.6, yscale=.7]
      \node (v7) at (0, -1) {};
      \draw (-2, 0) -- (2, 0);
      \foreach[count=\i from 2] \v/\p in {x_{1}/above, v_{1}/above, v_{2}/above right, v_{3}/above, x_{3}/above} {
        \node (v\i) at ({\i - 4}, 0) {};
        \draw (v\i) -- (v7);
      }
      \node (v1) at (0, .75) {};
      \draw (v4) -- (v1);      
    \end{tikzpicture}
    \caption{whipping top}\label{fig:whipping-top-unlabeled}
  \end{subfigure}
  \,
  \begin{subfigure}[b]{0.22\linewidth}
    \centering
    \begin{tikzpicture}[scale=.8]
      \draw (-1.5, 0) -- (0, 0) edge[dashed] (1.5, 0);
      \draw (0, 1.75) node {} -- (0, 1) node (c) {} -- (0, 0) node {};
      \foreach[count =\i] \y in {1, 3} {
        \node (u\i) at ({3 * \i - 4.5}, 0) {};
        \node (v\i) at ({2 * \i - 3}, 0) {};
        \draw (u\i) -- (v\i) -- (c);
      }
    \end{tikzpicture}
    \caption{\dag{}}\label{fig:dag} 
  \end{subfigure}
  \,
  \begin{subfigure}[b]{0.22\linewidth}
    \centering
    \begin{tikzpicture}[scale=.8]
      \draw[dashed] (-1.5, 0) -- (1.5, 0);
      \node (v3) at (0, 0) {};
      \foreach[count =\i, evaluate={\x=int(2*\i - 3);}] \a/\b in {1/1, p/3} {
        \node (u\i) at ({1.5 * \x}, 0) {};
        \node (v\i) at ({1. * \x}, 0) {};
        \node (c\i) at ({.35 * \x}, 1) {};
        \draw (u\i) -- (v\i) -- (c\i) -- (u\i);
      }
      \foreach \i in {1, 2} {
        \foreach \j in {1, 2, 3}
        \draw (c\i) -- (v\j);
      }
      \draw (0, 1.75) node (x) {} -- (c1) -- (c2) -- (x);
    \end{tikzpicture}
    \caption{\ddag{}}\label{fig:ddag}  
  \end{subfigure}
  \caption{Minimal graphs that are not interval graphs.  A~$\dag$ graph or a~$\ddag$ graph contains at least six vertices.}
  \label{fig:non-interval}
\end{figure}
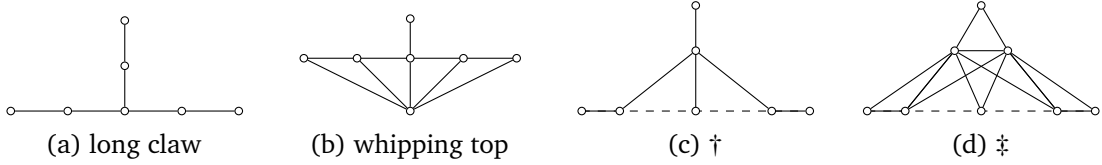

The same problem on circular-arc graphs, however, has been open for sixty years~\cite{hadwiger-64-combinatorial-geometry, klee-69-cag}.

\begin{problem}
  \label{prob:main_problem}
Determine the list of minimal forbidden induced subgraphs of circular-arc graphs.
\end{problem}

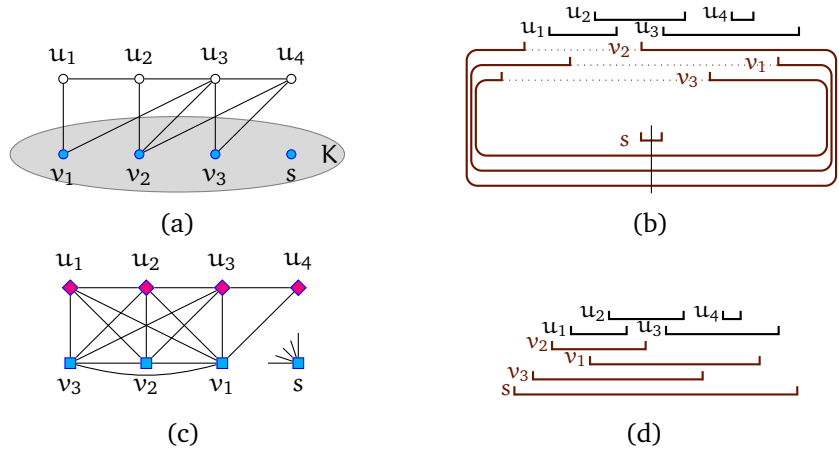
\begin{figure}[ht]
  \centering \small
  \begin{subfigure}[b]{0.35\linewidth}
    \centering
    \begin{tikzpicture}\draw[fill = gray!30, draw = gray] (2.5, 0) ellipse [x radius=2.2, y radius=0.5];
\draw (4, 0) node[gf-vertex] (s) {} ++(0, -7pt) node {$s$};
      \foreach \i in {1, 2, 3}
\draw (\i, 0) node[gf-vertex] (v\i) {} ++(0, -7pt) node {$v_\i$};
      \foreach \i in {1, 2, 3, 4}
      \node[gnf-vertex, "$u_\i$"] (u\i) at (\i, 1) {};
      \foreach[count=\i] \list in {{1, 3}, {2, 3, 4}, {3, 4}}
      \foreach \x in \list \draw (u\x) -- (v\i);

      \draw (u1) -- (u2) -- (u3) -- (u4);
      \node at (4.5, 0) {$K$};
    \end{tikzpicture}
    \caption{}\label{fig:ca-graph}
  \end{subfigure}
  \,
  \begin{subfigure}[b]{0.4\linewidth}
    \centering
    \begin{tikzpicture}[xscale=.3]
        \def\leftend{0}
        \def\rightend{15}
        \pgfmathsetmacro{\middle}{(\leftend + \rightend)/2}
        \def\bottomend{-.6}
        
      \begin{scope}[every path/.style={{|[right]}-{|[left]}}]
        \foreach[count=\i] \l/\r/\y in {3/6/4, 5/9/5, 8/14/4, 11/12/5}{
          \draw[thick] (\l-.02, \y/5) node[left, xshift=3pt, yshift=2pt] {\footnotesize $u_\i$} to (\r+.02, \y/5);
        }
      \draw[Sepia, thick] (\middle-.5, \bottomend) node[left] {\footnotesize $s$} to (\middle+.5, \bottomend);
      \end{scope}
        \foreach[count=\i] \l/\r/\y in {4/13/2, 2/7/3, 1/10/1}{
          \draw[Sepia, dotted] (\l-.02, \y/5) to (\r+.02, \y/5) node[left] {\footnotesize $v_\i$};
          \draw[{|[right]}-{|[left]}, Sepia, thick, rounded corners] (\r+.02, \y/5) -- ({\rightend + \y/5}, \y/5) --({\rightend + \y/5}, \bottomend-\y/5)--(\leftend - \y/5, \bottomend-\y/5)--(\leftend - \y/5, \y/5)--(\l-.02, \y/5);          
        }
\draw[] (\middle, \bottomend+.2) -- (\middle, \bottomend-.7);
      \end{tikzpicture}
    \caption{}\label{fig:normalized_model}
  \end{subfigure}

  \begin{subfigure}[b]{0.35\linewidth}
    \centering
    \begin{tikzpicture}\node[a-vertex, "$s$" below] (s) at (4, 0) {};
      \foreach[count=\i] \n in {3, 2, 1}
      \node[a-vertex, "$v_\n$" below] (v\n) at (\i, 0) {};
      \foreach[count=\i] \n in {1, 2, 3, 4}
      \node[b-vertex, "$u_\n$"] (u\n) at (\i, 1) {};
      \foreach[count=\i] \list in {{1, 2, 3, 4}, {1, 2, 3}, {1, 2, 3}}
      \foreach \x in \list \draw (u\x) -- (v\i);

      \draw (u1) -- (u2) -- (u3) -- (u4);
      \draw (v1) -- (v2) -- (v3);
      \draw[bend left=15] (v1) edge (v3);
      \foreach \i in {0, ..., 4}
      \draw (s) -- ++(\i*.1 - .4, \i*.1);
    \end{tikzpicture}
    \caption{}\label{fig:flipped-graph}
  \end{subfigure}
  \begin{subfigure}[b]{0.4\linewidth}
    \centering
    \begin{tikzpicture}[xscale=.25]
        \def\leftend{0}
        \def\rightend{15}
        \pgfmathsetmacro{\middle}{(\leftend + \rightend)/2}
        \def\bottomend{-.6}
        
      \begin{scope}[every path/.style={{|[right]}-{|[left]}}]
        \foreach[count=\i] \l/\r/\y in {3/6/4, 5/9/5, 8/14/4, 11/12/5}{
          \draw[thick] (\l-.02, \y/5) node[left, xshift=3pt, yshift=2pt] {\footnotesize $u_\i$} to (\r+.02, \y/5);
        }
      \draw[Sepia, thick] (\leftend, 0) node[left, xshift=3pt, yshift=2pt] {\footnotesize $s$} to (\rightend, 0);
        \foreach[count=\i] \l/\r/\y in {4/13/2, 2/7/3, 1/10/1}{
          \draw[Sepia, thick] (\l-.02, \y/5) node[left, xshift=3pt, yshift=2pt] {\footnotesize $v_\i$} to (\r+.02, \y/5);
        }
      \end{scope}
\end{tikzpicture}
    \caption{}\label{fig:interval_model}
  \end{subfigure}
  \caption{Illustration for McConnell's transformation.
    (a) A circular-arc graph $G$, where edges among the vertices in the clique $K$ (in the shadowed area) are omitted for clarity; (b) a normalized circular-arc model of~$G$; (c) the interval graph $G^K$, where edges incident to the universal vertex~$s$ are omitted for clarity; and (d) the interval model of~$G^K$ derived from (b) by flipping the arcs containing the center of the bottom.  
  }
  \label{fig:McConnell-transformation}
\end{figure}
 
The main goal of this paper is to make progress toward Problem~\ref{prob:main_problem}.
Our approach to Problem~\ref{prob:main_problem} has its roots in the work of McConnell~\cite{mcconnell-03-recognition-cag}, where the recognition of circular-arc graphs is reduced to the recognition of certain interval graphs.
Let us briefly describe the approach taken by McConnell.
A vertex $v \in V(G)$ is \emph{universal} in $G$ if $N[v] = V(G)$.
Let~$G$ be a circular-arc graph with no universal vertices and~$\mathcal{A}$ a fixed circular-arc model of~$G$.\footnote{Universal vertices have no impact on our problem.  However, the approach we take requires the nonexistence of universal vertices.}
By \emph{flipping} an arc~$[\lp, \rp]$ we replace it with arc~$[\rp, \lp]$ --- note that flipping is different from complementing, which ends with open arcs, and we prefer to stick to closed arcs and closed intervals for simplicity.
If we flip all arcs containing some fixed point of the circle, we end with an interval model~$\mathcal{I}$; see Figure~\ref{fig:McConnell-transformation} for an illustration.
A crucial observation of McConnell~\cite{mcconnell-03-recognition-cag} is that the resulting interval graph is decided by the set~$K$ of vertices whose arcs are flipped and not by the original circular-arc model, as long as it is normalized.
Thus, it makes sense to denote it by $G^K$.  We defer the definition of normalized models and~$G^K$ to Section~\ref{sec:McConnell-transformation}.
Note that~$K$ is a \emph{clique} of~$G$, i.e., all the edges among them are present in~$G$.

However, the graph~$G^K$ being an interval graph does not imply that~$G$ is a circular-arc graph.
The graph~$G$ is a circular-arc graph if and only if there is a clique~$K$ such that the graph~$G^K$ \begin{equation}
  \label{eq:sharp}
  \begin{array}{c}
    \text{admits an interval model in which for every pair of vertices~$v\in K$ and~$u\in V(G)\setminus K$, } \\
    \text{the interval for~$v$ doubly extends the interval for~$u$ if and only if~$uv\not\in E(G)$.}
  \end{array}
\tag{$\sharp$}
\end{equation}
In the above, the interval $I(v) = [\lp(v), \rp(v)]$ \emph{doubly extends} the interval $I(u) = [\lp(u),\rp(u)]$, written $I(u) \strictsubset I(v)$, if $\lp(v) < \lp(u) \leq \rp(u) < \rp(v)$.
Note that double extension is stronger than proper containment, which allows the two arcs to share a common endpoint.  The purpose of double extension is to ensure that flipping~$I(v)$ make it disjoint from~$I(u)$.
This approach can be summarized with the following theorem.
\begin{theorem}\label{thm:correlation}
  A graph $G$ with no universal vertices is a circular-arc graph if and only if there exists a clique~$K$ such that the graph~$G^{K}$ satisfies~\eqref{eq:sharp}.
\end{theorem}

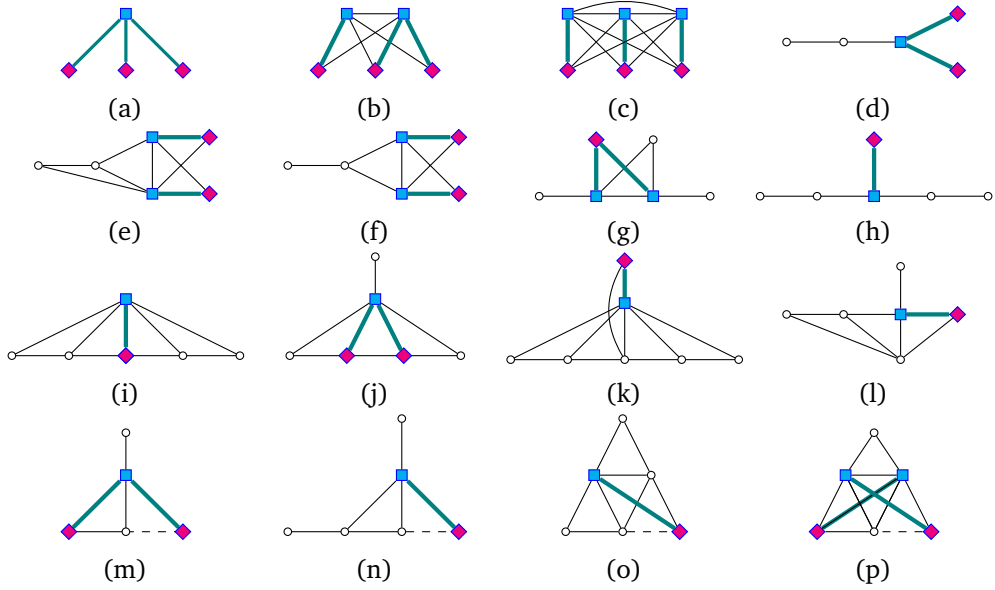
\begin{figure}[h]
  \centering \small
  \begin{subfigure}[b]{0.2\linewidth}
    \centering
    \begin{tikzpicture}[scale=.75]
      \node[a-vertex] (v) at (2, 1) {};
      \foreach \i in {1, 2, 3}
        \node[b-vertex] (u\i) at (\i, 0) {};
      \foreach \i in {1, 2, 3}
        \draw[very thick, teal] (v) -- (u\i);
    \end{tikzpicture}
    \caption{}
    \label{fig:claw-unlabeled}
  \end{subfigure}
  \begin{subfigure}[b]{0.2\linewidth}
    \centering
    \begin{tikzpicture}[scale=.75]
      \foreach \i in {1, 2}
      \node[a-vertex] (v\i) at (.5+\i, 1) {};
      \foreach \i in {1, 2, 3} {
        \node[b-vertex] (u\i) at (\i, 0) {};
}
      \draw (v1) -- (v2);
      \draw (v1) -- (u3) (v2) -- (u1);
      \draw[witnessed edge] (v1) -- (u1) (v2) -- (u2) (v2) -- (u3);
      \uncertain{u2}{v1};
    \end{tikzpicture}
    \caption{}
    \label{fig:double-claw-unlabeled}
  \end{subfigure}
  \begin{subfigure}[b]{0.2\linewidth}
    \centering
    \begin{tikzpicture}[scale=.75]
      \foreach \i in {1, 2, 3}
      \node[a-vertex] (v\i) at (\i, 1) {};
      \foreach \i in {1, 2, 3}
        \node[b-vertex] (u\i) at (\i, 0) {};
      \foreach \i in {1, 2, 3}
      \foreach \j in {1, 2, 3} {
        \ifnum \i=\j
          \def\c{witnessed edge}
        \else
          \def\c{}
        \fi
        \draw[\c] (v\i) -- (u\j);
        }
      \draw (v1) -- (v2) -- (v3);
      \draw[bend left=20] (v1) edge (v3);
    \end{tikzpicture}
    \caption{}
    \label{fig:triple-claw-unlabeled}
  \end{subfigure}
  \begin{subfigure}[b]{0.2\linewidth}
    \centering
    \begin{tikzpicture}[scale=.75]
      \draw (-2.,0) node[empty vertex] (x1) {} -- (2.,0) node[empty vertex] (x2) {};
      \foreach[count=\i] \l/\t in {v_1/empty , v/a-, v_2/empty }
      \node[\t vertex] (v\i) at ({\i*1-2}, 0) {};
      \node[b-vertex] (u) at (90:1) {};
      \draw[witnessed edge] (u) -- (v2);
    \end{tikzpicture}
    \caption{}
    \label{fig:p5+1-unlabeled}
  \end{subfigure}

  \begin{subfigure}[b]{0.2\linewidth}
    \centering
    \begin{tikzpicture}[scale=.75]
      \node[a-vertex] (v) at (3, 0) {};
      \foreach \i in {1, 2} {
        \node[rotate={180*\i}, b-vertex] (u\i) at (4, {1.5-\i}) {};
        \draw[witnessed edge] (v) -- (u\i);
      }
      \foreach[count=\i] \t in {b-, empty } {
        \node[empty vertex] (x\i) at ({\i},0) {};
      }
      \draw (v)--(x2)--(x1);    
    \end{tikzpicture}
    \caption{}
    \label{fig:fork-unlabeled}
  \end{subfigure}
  \begin{subfigure}[b]{0.2\linewidth}
    \centering
    \begin{tikzpicture}[scale=.75]
      \foreach[count=\i] \t in {b-, empty } {
        \node[empty vertex] (x\i) at ({\i},0) {};
      }
      \foreach \i in {1, 2} {
        \node[a-vertex] (v\i) at (3, {1.5-\i}) {};
        \node[b-vertex] (u\i) at (4, {1.5-\i}) {};
        \draw (x2) -- (v\i);
        \draw[witnessed edge] (v\i) -- (u\i);
      }
      \draw (x2)--(x1)  (v2) -- (v1) -- (u2);
      \uncertain{u1}{v2};
    \end{tikzpicture}
    \caption{}
    \label{fig:double-fork-unlabeled}
  \end{subfigure}
  \begin{subfigure}[b]{0.2\linewidth}
    \centering
    \begin{tikzpicture}[scale=.75]
      \foreach[count=\i] \t in {b-, empty } {
        \node[empty vertex] (x\i) at ({\i},0) {};
      }
      \foreach \i in {1, 2} {
        \node[a-vertex] (v\i) at (3, {1.5-\i}) {};
        \node[b-vertex] (u\i) at (4, {1.5-\i}) {};
        \draw (x2) -- (v\i);
        \draw[witnessed edge] (v\i) -- (u\i);
      }
      \draw (x2)--(x1)--(v2)  (v2) -- (v1) -- (u2);
      \uncertain{u1}{v2};
    \end{tikzpicture}
    \caption{}
    \label{fig:double-fork+1-unlabeled}
  \end{subfigure}
  \begin{subfigure}[b]{0.2\linewidth}
    \centering
    \begin{tikzpicture}[scale=.75]
      \draw (-2.,0) -- (2.,0);
      \node[a-vertex] (a) at (0, 1.) {};
      \node[b-vertex] (u) at (0, 0) {};
      \draw[witnessed edge] (a) -- (u);
      \node[b-vertex] (c) at (0, 0) {};
      \foreach[count=\j] \x in {-2, -1, 1, 2} {
        \draw (a) -- (\x, 0) node[empty vertex] (v\j) {};
      }
     \end{tikzpicture}
    \caption{}
    \label{fig:p5x1-unlabeled}
  \end{subfigure}

  \begin{subfigure}[b]{0.2\linewidth}
    \centering
    \begin{tikzpicture}[scale=.75]
      \draw (1, 0) -- (4, 0);
      \node[a-vertex] (a) at (2.5, 1) {};
      \draw (a) -- ++(0, .75) node[empty vertex] {};

      \foreach[count=\i] \t/\l in {empty /x_{1}, b-/u_{1}, b-/u_{2}, empty /x_{2}} {
        \draw (a) -- (\i, 0) node[\t vertex] (u\i) {};
      }
      \foreach \i in {2, 3}
      \draw[witnessed edge] (a) -- (u\i);
    \end{tikzpicture}
    \caption{}
    \label{fig:(p4+p1)*1-unlabeled}
  \end{subfigure}
  \begin{subfigure}[b]{0.2\linewidth}
    \centering
    \begin{tikzpicture}[scale=.75]
      \draw (1, 0) -- (4, 0);
      \foreach[count=\x from 2] \t/\l/\p in {b-/u/above, empty /x_{3}/below}  \node[\t vertex] (\l) at (\x, 1) {};
      
      \foreach[count=\i] \t/\l in {empty /x_{1}, a-/v_{1}, a-/v_{2}, empty /x_{2}} {
        \node[\t vertex] (u\i) at (\i, 0) {};
      }
      \foreach \i in {2, 3} \draw (u) edge[witnessed edge] (u\i) (u\i) -- (x_{3});
    \end{tikzpicture}
    \caption{}
    \label{fig:ab-wheel-unlabeled}
  \end{subfigure}
  \begin{subfigure}[b]{0.2\linewidth}
    \centering
    \begin{tikzpicture}[scale=.75]
      \draw (1, 1) -- (4, 1);

       \node[empty vertex] (v4) at (3, 0.2) {};
      \foreach[count=\i] \t/\v/\x in {empty /x_1/1, empty /x_2/2, a-/{\quad v}/3.5, b-/u/5} {
        \draw (v4) -- (\i, 1) node[\t vertex] (u\i) {};
      }

      \draw (u3) -- ++(0, .85) node[empty vertex] (x2) {};
      \draw[witnessed edge] (u3) -- (u4);
     \end{tikzpicture}
    \caption{}
    \label{fig:whipping-top-1-unlabeled}
  \end{subfigure}
  \begin{subfigure}[b]{0.2\linewidth}
    \centering
    \begin{tikzpicture}[scale=.75]
      \draw (-2.,0) -- (2.,0);
      \node[a-vertex] (a) at (0, 1.) {};
      \node[b-vertex] (u) at (0.75, 1) {};
      \foreach[count=\j] \x in {-2, -1, 0, 1, 2} {
        \draw (a) -- (\x, 0) node[empty vertex] (x\j) {};
      }
      \draw (a) edge[witnessed edge] (u) (u) edge (x3);
     \end{tikzpicture}
    \caption{}
    \label{fig:bent-whipping-top-unlabeled}
  \end{subfigure}

  \begin{subfigure}[b]{0.2\linewidth}
    \centering
    \begin{tikzpicture}[scale=.75]
      \draw (1, 0) -- (2, 0);
      \draw[dashed] (2, 0) -- (3, 0);
      \node[a-vertex] (a) at (2, 1) {};
      \draw (a) -- ++(0, .75) node[empty vertex] {};

      \foreach[count=\i] \t/\l in {b-/u_{1}, empty /x_{1}, b-/u_{2}} {
        \draw (a) -- (\i, 0) node[\t vertex] (u\i) {};
      }
      \foreach \i in {1, 3}
      \draw[witnessed edge] (a) -- (u\i);
    \end{tikzpicture}
    \caption{}
    \label{fig:dag+2e-unlabeled}
  \end{subfigure}
  \begin{subfigure}[b]{0.2\linewidth}
    \centering
    \begin{tikzpicture}[scale=.75]
      \node[a-vertex] (x) at (3., 1) {};
      \draw (x) -- ++(0, .75) node[empty vertex] {};

      \draw (1, 0) -- (3, 0);
      \foreach[count=\i] \t/\v in {empty /x_{1}, empty /x_{2}, empty /x_{3}, b-/u} {
        \node[\t vertex] (u\i) at (\i, 0) {};
      }
      \draw[dashed] (u3) -- (u4);
      \foreach \i in {2, 3} \draw (x) -- (u\i);
      \draw[witnessed edge] (x) -- (u4) {};
    \end{tikzpicture}
    \caption{}
    \label{fig:dag+e-unlabeled}
  \end{subfigure}
  \begin{subfigure}[b]{0.2\linewidth}
    \centering
    \begin{tikzpicture}[scale=.75]
      \node[empty vertex] (x1) at (2, 1.75) {};
      \node[empty vertex] (x2) at (2.5, 1) {};
      \foreach \i in {3, 4} 
      \node[empty vertex] (x\i) at (\i-2, 0) {};
      \node[a-vertex] (v) at (1.5, 1) {};
      \node[b-vertex] (x5) at (3., 0) {};

      \foreach \i in {1, 4, 5} \draw (x2) -- (x\i);
      \foreach \i in {1, ..., 5} \draw (v) -- (x\i);
      \draw (x3) -- (x4) (x4) edge[dashed] (x5);
      \draw[witnessed edge] (v) -- (x5) {};
    \end{tikzpicture}
    \caption{}
    \label{fig:ddag+e-unlabeled}
  \end{subfigure}
  \begin{subfigure}[b]{0.2\linewidth}
    \centering
    \begin{tikzpicture}[scale=.75]
      \node[empty vertex] (x0) at (0, 2.75) {};
      \node[empty vertex] (u) at (0, 1) {};
      \foreach \i in {1, 2} {
        \node[rotate={180*\i}, a-vertex] (v\i) at ({1.5-\i}, 2) {};
        \node[rotate={180*\i}, b-vertex] (x\i) at ({(1.5-\i)*2}, 1) {};
      }
      \draw (u) -- (x2) (v1) -- (v2) (x1) edge[dashed] (u);
      \foreach \i in {1, 2} {
        \foreach \j in {1, 2} 
        \draw (u)--(v\i)--(x\j);
        \draw (v\the\numexpr3-\i\relax) -- (x0);
        \draw[witnessed edge] (v\i) -- (x\the\numexpr3-\i\relax);
      }
      \draw (v1)--(x1);
    \end{tikzpicture}
    \caption{}
    \label{fig:ddag+2e-unlabeled}
  \end{subfigure}
  
  \begin{subfigure}[b]{0.2\linewidth}
    \centering
    \begin{tikzpicture}[scale=.75]
      \node[empty vertex] (x0) at (0, 2.75) {};
      \foreach \i in {1, 2} {
        \node[rotate={180*\i}, a-vertex] (v\i) at ({1.5-\i}, 2) {};
        \node[rotate={180*\i}, b-vertex] (u\i) at ({1.5-\i}, 1) {};
        \node[rotate={180*\i}, empty vertex] (x\i) at ({(1.5-\i)*2.75}, 2) {};
        \draw (x\i) -- (v\i) -- (x0);
        \draw[witnessed edge] (v\i) -- (u\i);
      }
      \draw (u1)--(v2)--(v1)--(u2)--(u1);
    \end{tikzpicture}
    \caption{}
    \label{fig:add-1-unlabeled}
  \end{subfigure}
  \begin{subfigure}[b]{0.2\linewidth}
    \centering
    \begin{tikzpicture}[scale=.75]
      \node[empty vertex] (x0) at (0, 2.75) {};
      \foreach \i in {1, 2} {
        \node[rotate={180*\i}, a-vertex] (v\i) at ({1.5-\i}, 2) {};
        \node[rotate={180*\i}, b-vertex] (u\i) at ({1.5-\i}, 1) {};
        \node[rotate={180*\i}, empty vertex] (x\i) at ({(1.5-\i)*2.75}, 1) {};
        \draw (x\i) -- (v\i) -- (x0);
      }
      \begin{pgfonlayer}{bg}    \draw (x1) -- (x2) (v1) -- (v2);
      \end{pgfonlayer}
      \foreach \i in {1, 2} {
        \foreach \j in {1, 2} 
        \draw (u\j)--(v\i)--(x\j);
        \draw[witnessed edge] (v\i) -- (u\i);
      }
    \end{tikzpicture}
    \caption{}
    \label{fig:add-2-unlabeled}
  \end{subfigure}

  \caption{Interval configurations whose absence in the interval graph $G^K$ asserts condition~\eqref{eq:sharp} 
  in the case when $G$ is $C_4$-free.
  The square nodes are ``in $K$,'' rhombus ``not in $K$,'' and round ``uncertain.''
  Between a square node and a rhombus node, a thick edge is ``in $G$,'' a thin edge is ``not in $G$,'' and it is ``uncertain'' otherwise (a solid line and a dotted line).
There are at least five vertices in \ref{fig:dag+2e-unlabeled}, and at least six vertices in \ref{fig:dag+e-unlabeled}, \ref{fig:ddag+e-unlabeled}, and~\ref{fig:ddag+2e-unlabeled}.
}
  \label{fig:forbidden-configurations-C_4-free}
\end{figure}

We are thus motivated to understand when~$G^{K}$ satisfies Condition~\eqref{eq:sharp}.
Toward this, we discover a weaker condition~\eqref{eq:star} that is equivalent to~\eqref{eq:sharp} when~$G$ is~$C_{4}$-free.  A graph is~\emph{$H$-free} if it does not contain~$H$ as an induced subgraph, and~$C_{4}$ is simple cycle on four vertices.
More importantly, we are able to identify all the forbidden configurations, as shown in Figure~\ref{fig:forbidden-configurations-C_4-free}, whose absence asserts the weaker condition.
    
\begin{definition}[Annotations]
  In a \emph{configuration}, each vertex has one of the three annotations: in~$K$, not in~$K$, or uncertain; each edge between a vertex ``in~$K$'' and a vertex ``not in~$K$'' has one of the three annotations: in~$G$, not in~$G$, or uncertain.
\end{definition}

Note that there is no annotation on edges between two vertices ``in~$K$,'' between two vertices ``not in~$K$,'' or incident to an ``uncertain'' vertex; they are always ``in~$G$,'' ``in~$G$,'' and ``uncertain,'' respectively.
We say that the graph~$G^{K}$ \emph{contains (an annotated copy of)} configuration~$F$ if there exists an isomorphism~$\varphi$ between~$F$ and an induced subgraph of~$G^{K}$ such that
  \begin{itemize}
  \item if $v$ is annotated ``in~$K$'' (resp., ``not in~$K$'') then $\varphi(v)\in K$ (resp., $\varphi(v)\not\in K$); and
  \item if an edge~$v u$ is annotated ``in~$G$'' (resp., ``not in~$G$'') then~$\varphi(v) \varphi(u)\in E(G)$ (resp., ~$\varphi(v) \varphi(u)\in E(G)$).
  \end{itemize}

The weaker condition~\eqref{eq:star} can be characterized by forbidden configurations.
\begin{theorem}\label{thm:main-1}
  A $C_{4}$-free graph~$G$ with no universal vertices is a circular-arc graph if and only if there exists a clique~$K$ such that~$G^{K}$ \begin{equation}
    \label{eq:natural}
    \text{is an interval graph and it does not contain any configuration shown in Figure~\ref{fig:forbidden-configurations-C_4-free}.}
    \tag{$\vDash_{c}$}
  \end{equation}
\end{theorem}

We have a stronger statement on chordal graphs; i.e., hole-free graphs.
Every chordal graph has a simplicial vertex (i.e., a vertex whose closed neighborhood forms a clique), and its closed neighborhood can be used as the clique~$K$.

\begin{theorem}\label{thm:main-chordal}
   A chordal graph~$G$ with no universal vertices is a circular-arc graph if and only if for every simplicial vertex~$s$, the graph~$G^{N[s]}$ is an interval graph and it does not contain any annotated copy of the configurations shown in Figure~\ref{fig:forbidden-configurations-C_4-free}.
\end{theorem}

The characterization of Helly circular-arc graphs, an important subclass of circular-arc graphs, is also open. 
A graph is a \emph{Helly circular-arc graph} if it admits a \emph{Helly circular-arc model}, where the arcs for every maximal clique have a shared point.
In Figure~\ref{fig:normal-and-helly}, e.g., the first model is Helly, and the second is not.
All interval models are Helly, and hence all interval graphs are Helly circular-arc graphs.
Our approach also applies to Helly circular-arc graphs~\cite{cao-24-split-cag, cao-24-cag-iii-chordal}.  Indeed, it is significantly simpler for this subclass: every maximal clique can be used as the clique~$K$.

\begin{theorem}\label{thm:main-helly}
A $C_{4}$-free graph~$G$ is a Helly circular-arc graph if and only if for every maximal clique~$K$ of~$G$, the graph~$G^{K}$ satisfies~\eqref{eq:natural}.
\end{theorem}

Note that the sufficiency of Theorem~\ref{thm:main-helly} does not hold when there are~$C_{4}$'s; e.g.,~$\overline{p K_{2}}$ is not a Helly circular-arc graph, though for every maximal clique~$K$ of~$G$, the graph~$G^{K}$ satisfies~\eqref{eq:natural}.

Theorem~\ref{thm:main-1} strengthens a previous result in~\cite{cao-24-split-cag}, which characterizes~$G^{K}$ when~$G$ is a split graph.\footnote{A graph is a split graph if its vertex set can be partitioned into a clique and an independent set.}
In~\cite{cao-24-split-cag}, we prove the characterization using the modular decomposition tree of~$G^{K}$.
In the present paper, we use the PQ-tree data structure~\cite{BoothLueker76}, a refinement of modular decomposition trees of interval graphs.
An interval graph of order~$n$ has at most~$n$ maximal cliques, and they can be arranged as clique paths~\cite{fulkerson-65-interval-graphs}.
Booth and Lueker~\cite{BoothLueker76} introduced a data structure PQ-tree to represent all clique paths.
We try to build a special clique path, which implies an interval model satisfying~\eqref{eq:star}, bottom-up along with the PQ-tree of~$G^{K}$.
It fails only when it identifies a configuration shown in Figure~\ref{fig:forbidden-configurations-C_4-free}.
Thus, our proof is constructive and robust.

The characterization of~$G^{K}$ when~$G$ is a split graph enabled us to identify all minimal forbidden induced subgraphs of circular-arc graphs within the class of split graphs~\cite{cao-24-split-cag}.
In companion papers~\cite{cao-24-cag-iii-chordal, cao-24-cag-iv-c4-free}, we use Theorem~\ref{thm:main-chordal} and~\ref{thm:main-1} to identify all
minimal forbidden induced subgraphs of circular-arc graphs within chordal graphs and, respectively, $C_{4}$-free graphs.
The proofs imply polynomial-time certifying recognition algorithms for~$C_{4}$-free circular-arc graphs and chordal circular-arc graphs.

Our ultimate goal is to fully characterize Condition~\eqref{eq:sharp}, 
which will help us solve Problem~\ref{prob:main_problem}.
A closely related problem is certifying recognition of circular-arc graphs.

\begin{problem}
  \label{prob:alg_problem}
  Given a graph~$G$, either determine that it is a circular-arc graph, or identify a minimal forbidden induced subgraph.
\end{problem}

A natural approach to Problem~\ref{prob:alg_problem} is outlined as follows.
\begin{enumerate}
\item[(1)] Find a ``suitable'' clique $K$ in $G$, which can be used to construct the graph $G^K$.
\item[(2)] Check whether $G^K$ is an interval graph.
  If not, find a minimal not interval subgraph $F$ in $G^K$ and based on~$F$ and~$K$ extract from $G$ a minimal forbidden induced subgraph.
\item[(3)] Check whether the interval graph $G^K$ admits an interval model which satisfies Condition~\eqref{eq:sharp}.
If $G^K$ has such a model, construct a circular-arc model of $G$ by flipping the intervals from~$K$.
Otherwise, find an interval forbidden configuration $F$ in $G^K$ violating Condition~\eqref{eq:sharp}
and then, based on $F$ and~$K$, extract from $G$ a minimal forbidden induced subgraph.
\end{enumerate}

\section{McConnell's transformation}
\label{sec:McConnell-transformation}

A graph $G$ is a \emph{circular-arc} graph if there exists a mapping $A$ from $V(G)$ 
to the set of arcs of a circle such that for every $u,v \in V(G)$ we have $uv \in E(G)$ if and only if 
$A(u) \cap A(v) \neq \emptyset$.
If the above holds, then the set $\mathcal{A} = \big{\{} A(v) \mid v \in V(G) \big{\}}$ is called a \emph{circular-arc model} of $G$.
A~circular-arc model~$\mathcal{A}$ of~$G$ is \emph{normalized} if
\begin{enumerate}[I.]
\item all the arcs in $\mathcal{A}$ have different endpoints;
\end{enumerate}
and the following hold for every pair of adjacent vertices~$v_{1}$ and~$v_{2}$:
\begin{enumerate}[I.]
  \setcounter{enumi}{1}
\item if~$N[v_{1}] \subseteq N[v_{2}]$, then~$A(v_{1})\subseteq A(v_{2})$;
\item the arcs~$A(v_{1})$ and~$A(v_{2})$ cover the circle whenever $N(v_{1})\cup N(v_{2}) = V(G)$ and~$v_{1} v_{2}$ is not an edge of a $C_{4}$ (an induced cycle on four vertices).
\end{enumerate}
Two vertices~$v_1$ and~$v_2$ are \emph{twins} in~$G$ if $N[v_1] = N[v_2]$.
It is known that, if $G$ has no universal vertices and no twins, every circular-arc model of $G$ can be turned into a normalized model by possibly extending some arcs of this model~\cite{spinrad-88-case-1, hsu-95-independent-set-cag}.
It is pedestrian to generalize the notion of normalized models to graphs with twins by replacing condition (I) with the following.
\begin{enumerate}[I$'$.]
\item Two arcs~$A(v_{1})$ and~$A(v_{2})$ share an endpoint if and only if~$v_1$ and~$v_2$ are twins.  \end{enumerate}
Note that condition~(II) forces~$A(v_{1}) = A(v_{2})$ when~$v_1$ and~$v_2$ are twins.
On the other hand, we must exclude universal vertices, which are inherently incompatible with normalized circular-arc models.

Following McConnell~\cite{mcconnell-03-recognition-cag}, for a graph $G$ with no universal vertex and a clique $K$ of $G$ (not necessarily maximal), we introduce an \emph{auxiliary graph}~$G^{K}$ with the vertex set~$V(G^K) = V(G)$ and the edge set $E(G^K)$ defined as follows.
\begin{itemize}
 \item Two vertices $u,u' \in V(G) \setminus K$ are adjacent in $G^K$ if and only if $u u' \in E(G)$.
 \item Two vertices $v,v' \in K$ are adjacent in $G^K$ if and only if $N_G[v] \cup N_G[v'] \neq V(G)$
 (there is $w \in V(G) \setminus K$ nonadjacent to $v$ and $v'$ in $G$) or $vv'$ is an edge of $C_4$ in $G$.
 \item Two vertices $v \in K$ and~$u \in V(G)\setminus K$ are adjacent in $G^K$ if and only if $vu \notin E(G)$ 
   or $vu \in E(G)$ and $N_{G}[u]\not\subseteq N_{G}[v]$ (there is $w \in V(G) \setminus K$ adjacent to $u$ and not adjacent to $v$ in $G$).
\end{itemize}
See Figure~\ref{fig:McConnell-transformation} for an illustration. 
It is worth mentioning that for two vertices $v \in K$ and~$u \in V(G)\setminus K$, the condition can be stated as~$N_{G}[u]\not\subseteq N_{G}[v]$.
In Figure~\ref{fig:McConnell-transformation}c, the edge~$v_{1} v_{2}$ exists because~$v_{1} v_{2} u_{2} u_{1}$ is a~$C_{4}$ in~$G$; the edges~$v_{1} v_{3}$ and~$v_{2} v_{3}$ exist because of the vertices~$u_{2}$ and~$u_{1}$, respectively; while the edge~$v_{2}u_{2}$ exists because of the vertex~$u_{1}$.

The following observation is immediate from the construction of~$G^{K}$.
\begin{proposition}
\label{obs:uv_not_in_G}
For every pair of vertices~$v\in K$ and~$u \in V(G) \setminus K$, if $uv \notin E(G)$ then
$N_{G^K}[u] \subseteq N_{G^K}[v]$.
\end{proposition}
\begin{proof}
Let $u' \in V(G) \setminus K$ be such that $u' \in N_{G^K}(u)$.
Then, we have $u' \in N_{G^K}(v)$ as we have $u'u \in E(G)$ and $uv \notin E(G)$.

Let $v' \in K$ be such that $v' \in N_{G^K}(u)$.
We will show that $v' \in N_{G^K}(v)$.
If $uv' \notin E(G)$, then we have $vv' \in E(G^K)$ as $uv \notin E(G)$ and $uv' \notin E(G)$.
So, we assume $uv' \in E(G)$.
Since $v' \in N_{G^K}(u)$ and $uv' \in E(G)$, there is $u' \in V(G) \setminus K$ such 
that $uu' \in E(G)$ and $u'v' \notin E(G)$.
If $u'v \notin E(G)$, then $vv' \in E(G^K)$ as $u'v' \notin E(G)$ and $u'v \notin E(G)$.
Otherwise, $\{u,u',v,v'\}$ induces $C_4$ in $G$, and hence $vv' \in E(G^K)$.
\end{proof}

A graph $H$ is an \emph{interval graph} if there is a mapping $I$ from $V(H)$ to the set of closed intervals in the real line such that for every $u,v \in V(H)$ we have $I(u) \cap I(v) \neq \emptyset$ 
if and only if $uv \in E(H)$.
Such a set $\mathcal{I} = \big{\{}I(v) = [\lp(v), \rp(v)] \mid v \in V(H)\big{\}}$
is called an \emph{interval model} of~$H$.
If for each of~$n$ (not necessarily distinct) left endpoints of the $n$ intervals, we take the set of vertices whose intervals contain this point, then we end with $n$ cliques. 
We leave it to the reader to verify that they include all the maximal cliques of~$H$. 
If we list the distinct maximal cliques from left to right, sorted by the endpoints that we use to define these cliques, then we can see that for any~$v\in V(H)$, the maximal cliques of~$H$ containing $v$ appear consecutively.
We say that such a linear arrangement of maximal cliques is a \emph{clique path} of~$H$.
On the other hand, given a clique path~$\sigma = \langle K_{1}, K_{2}, \ldots, K_{\ell}\rangle$ for an interval graph~$H$ with $\ell$ maximal cliques, for each vertex~$v$ we can define an interval~$[\lk_{\sigma}(v), \rk_{\sigma}(v)]$, where $\lk_{\sigma}(v)$ and~$\rk_{\sigma}(v)$ are the indices of the first and, respectively, last maximal cliques containing $v$ 
(we drop the subscript $\sigma$ if it is clear from the context). 
One may easily see that they define an interval model for~$H$; see, e.g., Figure~\ref{fig:clique-path-interval-model}. 
Therefore, clique paths and interval models
are interchangeable, and when we illustrate clique paths, we always use the way in Figure~\ref{fig:clique-path-interval-model}b.

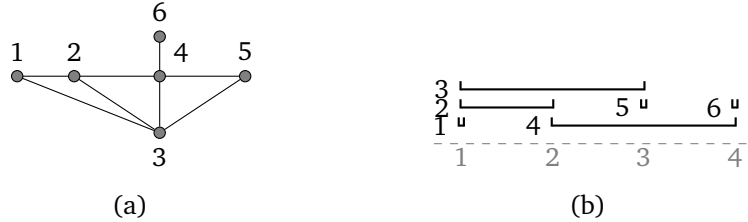
\begin{figure}[h]
  \centering\small
  \begin{subfigure}[b]{.35\linewidth}
    \centering
    \begin{tikzpicture}[scale=.75]
      \draw (1, 1) -- (5, 1);
      \node[filled vertex, "$3$" below] (v4) at (3.5, 0) {};
      \foreach[count=\i] \v/\x in {1/1, 2/2, /3.5, 5/5} {
        \draw (v4) -- (\x, 1) node[filled vertex, "$\v$"] (u\i) {};
      }
      \node[above right = 0mm of u3] {$4$};
      \draw (u3) -- ++(0, .7) node[filled vertex, "$6$"] (x2) {};
    \end{tikzpicture}
    \caption{}
  \end{subfigure}  
  \quad
  \begin{subfigure}[b]{.35\linewidth}
    \centering
    \begin{tikzpicture}[scale=1.2]
      \begin{scope}[every path/.style={{|[right]}-{|[left]}}]
        \foreach \x in {1, ..., 4} \node[gray] at (\x-1, -.15) {$\x$};
        \foreach[count=\i] \l/\r/\y/\c in {-.02/.02/1/, 0/1/2/, 0/2/3/, 1/3/1/, 1.98/2.02/2/, 2.98/3.02/2/}{
          \draw[\c, thick] (\l-.02, \y/4) node[left] {$\i$} to (\r+.02, \y/4);
        }
      \end{scope}
      \draw[dashed, gray] (-.3, 0) -- (3.2, 0);
    \end{tikzpicture}
    \caption{}
  \end{subfigure}  
  \caption{(a) An interval graph and (b) its clique path represented as an interval model.}
  \label{fig:clique-path-interval-model}
\end{figure}

\begin{lemma}\label{lem:construction}
  Let $G$ be a circular-arc graph with no universal vertices, and let~$K$ be a clique of~$G$.
  If there is a normalized circular-arc model of~$G$ in which a point is covered by and only by arcs for vertices in~$K$, then the graph~$G^{K}$ satisfies~\eqref{eq:sharp}.
\end{lemma}
\begin{proof}
  Let $\mathcal{A}$ be a normalized circular-arc model of~$G$ in which a point $P$ is covered by and only by arcs for vertices in~$K$. 
  Note that
a pair of twins of~$G$ is both in either~$K$ or~$V(G) \setminus K$.
For each vertex~$x\in V(G)$, denote by~$\lp(x)$ and~$\rp(x)$, respectively, the counterclockwise and clockwise endpoints of the arc~$A(x)$.
  We produce an interval model of $G^K$ by setting
  \[
    I(x) =
    \begin{cases}
      [\rp(x), \lp(x)] & \text{if } x\in K,
      \\
      [\lp(x), \rp(x)] & \text{if } x\in V(G)\setminus K.
    \end{cases}
  \]
  That is, we flip the arcs containing the point $P$. 
  See Figure~\ref{fig:McConnell-transformation} for an illustration.
  
  We show that (a)~a pair of vertices~$x$ and~$y$ are adjacent in~$G^{K}$ if and only if $I(x)\cap I(y)\ne \emptyset$ (i.e., $\{I(x) \mid x \in V(G)\}$ is an interval model for~$G^{K}$); and (b)~for all~$x\in K$ and~$y\in V(G)\setminus K$,
  \[
    I(y)\strictsubset I(x) \Leftrightarrow x y\not\in E(G).
  \]
  
  First, if both $x$ and~$y$ are in~$V(G)\setminus K$, then $I(x) \cap I(y) = A(x)\cap A(y)$, which is empty if and only if $x y\not\in E(G)$ and~$x y\not\in E(G^{K})$.

  Second, suppose that $x, y\in K$.  
  Recall that $x y\in E(G^{K})$ if and only if $N_{G}[x]\cup N_{G}[y] \ne V(G)$ or
  $xy$ is an edge of $C_4$ in $G$.
  If $x y\notin E(G^{K})$, then $N_{G}[x]\cup N_{G}[y] = V(G)$ and $uv$ is not an edge of~$C_{4}$ in $G$, 
  and hence $A(x)$ and~$A(y)$ cover the circle as~$\mathcal{A}$ is normalized. 
  Hence $I(x)$ and~$I(y)$ are disjoint. 
  Suppose that $x y\in E(G^{K})$.
  If there exists a vertex~$w\in V(G)\setminus K$ that is adjacent to neither of $x$ nor $y$ in~$G$, then~$I(w) = A(w) \subseteq I(x)\cap I(y)$.
  If~$xy$ is an edge of a~$C_{4}$ in~$G$, let the other two vertices be~$w$ and~$w'$.
  Note that $w, w'\not\in K$.
  Thus,
  \[
    \emptyset \ne I(w)\cap I(w') = A(w)\cap A(w') \subseteq I(x)\cap I(y).
  \]
  
  Finally, suppose that $x\in K$ and~$y\in V(G)\setminus K$.
  If $x y\not\in E(G)$, then $I(y)\strictsubset I(x)$ and~$x y\in E(G^{K})$.
  In the rest, $x y\in E(G)$.
  Since $A(x)\cap A(y)\ne \emptyset$, $I(y)\not\strictsubset I(x)$.
  Recall that
   $x y\in E(G^{K})$ if and only if there exists a vertex~$w\in N_{G}(y)\setminus N_{G}(x)$.
  If such a vertex~$w$ exists, then $I(y) \cap I(w)\subseteq I(w) \strictsubset I(x)$, and hence~$I(x) \cap I(y)\ne \emptyset$.
  Otherwise, $N_{G}[y]\subsetneq N_{G}[x]$.
  Since the model~$\mathcal{A}$ is normalized, $A(y)\strictsubset A(x)$, which means that $I(x)\cap I(y)=\emptyset$.
\end{proof}

We are now ready to prove Theorem~\ref{thm:correlation}.
\begin{proof}[Proof of Theorem~\ref{thm:correlation}]
Suppose $G$ is a circular-arc graph with no universal vertices.
Let $\mathcal{A}$ be a normalized model of $G$ and let $P$ be any point of the circle. 
Let $K = \{v \in V(G) \mid P \in A(v)\}$. 

Let $\mathcal{I} = \big{\{}I(v) \mid v \in V(G)\big{\}}$ be a set of intervals obtained by flipping every arc from $\mathcal{A}$ containing the point $P$.
By Lemma~\ref{lem:construction}, $\mathcal{I}$ is an interval model of $G^K$ satisfying Condition~\eqref{eq:sharp}.

In the rest, we show sufficiency.  
Suppose that $\mathcal{I} = \big{\{} [\lp(x), \rp(x) ] \mid x\in V(G) \big{\}}$ is an interval model of~$G^{K}$ satisfying Condition~\eqref{eq:sharp}.
We may assume that all the endpoints in~$\mathcal{I}$ are positive, and hence no interval contains the point~$0$.
We claim that the following arcs on a circle of length~$\ell + 1$, where $\ell$ denotes the maximum of the $2n$ endpoints in~$\mathcal{I}$, gives a circular-arc model of~$G$:
  \[
    A(x) =
    \begin{cases}
      [\rp(x), \lp(x)] & \text{if } x\in K,
      \\
      [\lp(x), \rp(x)] & \text{if } x\in V(G)\setminus K.
    \end{cases}
  \]
All the arcs for vertices in~$K$ intersect because they cover the point~$0$.
On the other hand, note that $G^{K} - K$ is isomorphic to~$G - K$, while $I(x) = A(x)$ for all $x\in V(G)\setminus K$.
For a pair of vertices~$x\in K$ and~$y\in V(G)\setminus K$,
by assumption, $x y\not\in E(G)$ if and only if $I(y)\strictsubset I(x)$, which is equivalent to~$A(x)\cap A(y) = \emptyset$.
\end{proof}

We now consider clique paths of~$G^K$ when it satisfies~\eqref{eq:sharp}.
In a clique path of $G^K$, by definition, for any two vertices $x, y \in V(G)$, we have 
\[
  N_{G^K}[x] \subseteq N_{G^K}[y] \quad \text{if and only if} \quad
  \lk(y) \leq \lk(x) \leq \rk(x) \leq \rk(y).
\]
Consider a pair of vertices $v \in K$ and $u \in V(G) \setminus K$.
If $uv \not\in E(G)$, then $N_{G^K}[u] \subseteq N_{G^K}[v]$ by Proposition~\ref{obs:uv_not_in_G}, 
and hence $\lk(v) \leq \lk(u) \leq \rk(u) \leq \rk(v)$ in any clique path of $G^K$.
If $uv \in E(G)$ and $N_{G^K}[u] \not \subseteq N_{G^K}[v]$, then $\lk(u) > \lk(v)$ or $\rk(v) < \rk(u)$ in any clique path of $G^K$.
Most importantly, if $uv \in E(G)$ and $N_{G^K}[u] \subseteq N_{G^K}[v]$, then the clique path of any interval model~$\mathcal{I}$ of $G^K$ admitting Condition~\eqref{eq:sharp} must satisfy $\lk(u)=\lk(v)$ or $\rk(v) =\rk(u)$ as $I(u) \not \strictsubset I(v)$.

Let
\[
  \mathcal{P} = \big{\{}(v,u) \in K \times V(G) \setminus K  \mid N_{G^K}[u] \subseteq N_{G^K}[v] \text{ and } vu \in E(G)\big{\}}.
\]
Summing up the previous paragraph, if an interval model of $G^K$ admits Condition~\eqref{eq:sharp}, then its corresponding clique path satisfies Condition~\eqref{eq:star}:
\begin{equation}
\label{eq:star}
\begin{array}{c}
\text{For every pair $(v, u) \in \mathcal{P}$, it holds~$\lk(u) = \lk(v)$ or $\rk(v) = \rk(u)$}.
\end{array}
\tag{$\vDash$}
\end{equation}
We say the graph~$G^{K}$ satisfies Condition~\eqref{eq:star} if it is an interval graph and it admits a clique path satisfying Condition~\eqref{eq:star}.

\begin{observation}
\label{obs:sharp-implies-star}
Let $G$ be a graph with no universal vertices and let $K$ be a clique of $G$.
If~$G^K$ satisfies Condition~\eqref{eq:sharp} then $G^K$~satisfies Condition~\eqref{eq:star}. 
\end{observation}

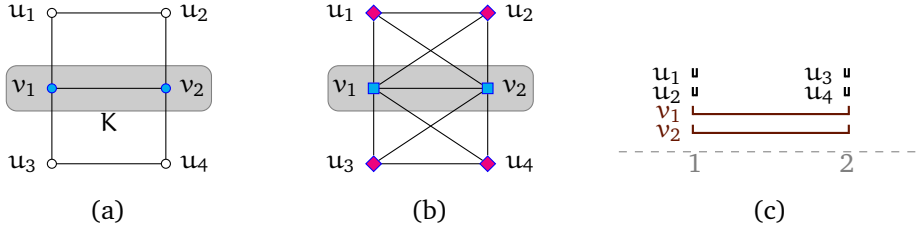
\begin{figure}[h]
  \centering \small
  \begin{subfigure}[b]{0.25\linewidth}
    \centering
    \begin{tikzpicture}[xscale=1.5]
    \draw[region] (.6, 0.7) rectangle (2.4, 1.3);
    \draw (1, 0) grid (2, 2);
    \foreach \i/\p in {1/left, 2/right} {
\node[gf-vertex, "$v_{\i}$" \p](v\i) at (\i, 1) {};
      \foreach \j in {0, 1} {
        \pgfmathsetmacro{\x}{int(\numexpr \i+\j*2)}\node[gnf-vertex, "$u_{\x}$" \p] (u\x) at (\i, 2 - \j*2) {};
      }
    }
    \node at (1.5, .55) {$K$};
    \end{tikzpicture}
    \caption{}\label{fig:domino}
  \end{subfigure}
  \,
  \begin{subfigure}[b]{0.25\linewidth}
    \centering
    \begin{tikzpicture}[xscale=1.5]
    \draw[region] (.6, 0.7) rectangle (2.4, 1.3);
\foreach \y in {0, 1, 2} 
    \draw (1, \y) -- (2, \y);
    \foreach[count=\i] \xd in {1, -1} {
      \foreach \yd in {1, -1} 
      \draw (\i, 1) -- ++(\xd, \yd);
      \draw[witnessed edge] (\i, 0) -- (\i, 2);
    }
  
    \foreach[count=\i] \xd in {-1, 1} {
        \draw (\i, 1) node[a-vertex] (v\i) {} ++ ({\xd/5}, 0) node {$v_{\i}$};
        \foreach[count=\j] \yd in {-1, 1} {
          \draw (\i, 1 + \yd) node[b-vertex] {} ++ ({\xd/5}, 0) node {$u_{\inteval{\i+\j*2}}$};
      }
    }
    \end{tikzpicture}
    \caption{}
  \end{subfigure}
  \begin{subfigure}[b]{0.3\linewidth}
    \centering
    \begin{tikzpicture}[xscale=2]
      \begin{scope}[every path/.style={{|[right]}-{|[left]}}]
        \foreach \x in {1, ..., 2} \node[gray] at (\x-1, -.15) {$\x$};
        \foreach \i in {1, 2}
          \draw[Sepia, thick] (-.02, {(3 -\i)/4}) node[left] {$v_\i$} to (1.02, {(3 -\i)/4});
          \foreach \i in {1, 2} {
            \foreach \j in {0, 1} 
            \draw[thick] (\j-.02, {(5-\i)/4}) node[left] {$u_{\the\numexpr \i+\j*2\relax}$} to (\j+.02, {(5-\i)/4});
        }
      \end{scope}
      \draw[dashed, gray] (-.5, 0) -- (1.5, 0);
    \end{tikzpicture}
    \caption{}
    \label{fig:sharp-versus-star-aligned-model}
  \end{subfigure}
  \caption{Illustration for $G^K$ satisfying~\eqref{eq:star} but not satisfying~\eqref{eq:sharp}.
   (a) The graph $G$, where $K = \{v_{1}, v_{2}\}$; (b) the interval graph $G^K$; and (c)~ an interval model~of~$G^K$.
 }
\label{fig:sharp-versus-star}
\end{figure}

However, Condition~\eqref{eq:star} turns out to be weaker than Condition~\eqref{eq:sharp}.
Figure~\ref{fig:sharp-versus-star} shows an example of
$G^K$ satisfying Condition~\eqref{eq:star}, witnessed by the clique path shown in Figure~\ref{fig:sharp-versus-star-aligned-model}, but not Condition~\eqref{eq:sharp}.
Indeed, suppose we want to transform the model in Figure~\ref{fig:sharp-versus-star-aligned-model} to obtain an interval
model satisfying~\eqref{eq:sharp}.
Since $u_1v_2 \notin E(G)$ and $u_2v_1 \notin E(G)$, 
we must have $\lp(v_1) < \lp(u_2)$ and $\lp(v_2) < \lp(u_1)$, 
and since $u_1v_1 \in E(G)$ and $u_2v_2 \in E(G)$, we must have $\lp(u_1) \leq \lp(v_1)$ and 
$\lp(u_2) \leq \lp(v_2)$.
One can easily check that we can not arrange the left endpoints of the intervals for $u_1,v_1,u_2,v_2$ to fulfil 
these four constraints.
This is not surprising as $G$ is not a circular-arc graph.

The difference between Condition~\eqref{eq:star} and Condition~\eqref{eq:sharp} lies in~$C_{4}$'s; e.g., note that~$u_1 v_1 v_2 u_2$ is a~$C_4$ in Figure~\ref{fig:domino}.
They are equivalent when~$G$ is~$C_4$-free.

\begin{lemma}
  \label{lem:star-sharp-equivalence-in-C_4-free}
  Let $G$ be a $C_4$-free graph with no universal vertices and let $K$ be a clique in $G$. 
  Then~$G^K$ satisfies~\eqref{eq:sharp} if and only if $G^K$ satisfies~\eqref{eq:star}.
\end{lemma}
\begin{proof}
  We have seen the necessity in Observation~\ref{obs:sharp-implies-star}, and the proof is focused on sufficiency.
  Suppose that~$\langle K_1,\ldots K_t \rangle$ is a clique path of~$G^K$ that satisfies~\eqref{eq:star}.
  We show that we can transform it into a model $\mathcal{I}$ satisfying Condition~\eqref{eq:sharp}.

For notational convenience, we introduce empty sets~$K_{0}$ and~$K_{t+1}$ as sentinels.
Then, for every~$i$ with~$1\le i \le t$, both the sets $K_i \setminus K_{i-1}$ and $K_i \setminus K_{i+1}$ are nonempty.
For every $i = 1, \ldots, t$ and every vertex $v \in (K_{i} \setminus K_{i-1}) \cap K$, we define the set
\[
  L(v) = \big{\{} u \in N_{G}(v) \setminus K \mid u \in (K_{i} \setminus K_{i-1}) \big{\}}.
  \]
We claim that 
\begin{equation}
\label{eq:C_4_free_cond_left}
\text{for every $v_{1}, v_{2} \in (K_i \setminus K_{i-1}) \cap K$,
one of~$L(v_{1})$ and~$L(v_{2})$ is a subset of the other.
}
\end{equation}
Suppose for contradiction that there are~$u_{1} \in L(v_{1})\setminus L(v_{2})$ and~$u_{2}\in L(v_{2})\setminus L(v_{1})$.
Then~$u_{1} v_{1}, u_{2} v_{2} \in E(G)$ and~$u_{1} v_{2}, u_{2} v_{1} \notin E(G)$.
Note that~$v_{1} v_{2} \in E(G)$ since $v_{1}, v_{2} \in K$, and~$u_{1} u_{2} \in E(G)$ since $u_{1}, u_{2} \in K_i \setminus K$.
Thus, the set $u_{1} v_{1} v_{2} u_{2}$ is a~$C_{4}$ in $G$, which cannot be the case.

Accordingly, we may order the left endpoints of the intervals from $K_i \setminus K_{i-1}$ such that:
For every pair of~$v \in (K_{i} \setminus K_{i-1}) \cap K$ and~$u \in (K_{i} \setminus K_{i-1}) \setminus K$, we set $\lp(u) < \lp(v)$ if $u v \in E(G)$ (i.e., $u\in L(v)$), or $\lp(v) > \lp(u)$ otherwise.
By \eqref{eq:C_4_free_cond_left}, this is a partial order.
We extend it into a total order, and put them in this order in the range~$(i - {1\over 3}, i)$.

We deal with the right endpoints in a symmetric way.
Suppose that $1\le i \le t$.  For every pair of~$v \in (K_{i} \setminus K_{i+1}) \cap K$ and~$u \in (K_{i} \setminus K_{i+1}) \setminus K$, we set $\rp(u) > \rp(v)$ if $u v \in E(G)$, or $\rp(v) < \rp(u)$ otherwise.
We extend this partial order into a total order, and put them in this order in the range~$(i, i + {1\over 3})$.

Note that this model is an interval model for~$G^{K}$: the intervals for two vertices intersect if and only if there is~$i$ such that they are both in~$K_{i}$.
Finally, we consider every pair~$v \in K$ and~$u \in V(G) \setminus K$ such that $N_{G^K}(u) \subseteq N_{G^K}(v)$.
Since $G^K$ satisfies~\eqref{eq:star}, if~$u v\in E(G)$, then~$\lk(v) = \lk(u)$ or~$\rk(u) = \rk(v)$.
We may consider~$\lk(v) = \lk(u)$ and the other is symmetric.
By the ordering, $\lp(u) < \lp(v)$, and hence~$I(u)\not\strictsubset I(v)$.
On the other hand, if~$u v\in E(G)$, then~$\lk(v) \le \lk(u)$ and~$\rk(u) \le \rk(v)$.
By the ordering, we always have~$\lp(v) < \lp(u) < \rp(u) < \rp(v)$.
Thus, the interval model of $G^K$ we have constructed satisfies Condition~\eqref{eq:sharp}.
\end{proof}

We are able to fully characterize Condition~\eqref{eq:star}.
The proof is deferred to the next section.

\begin{lemma}
\label{thm:star} 
Let $G$ be a graph with no universal vertices and let~$K$ be a clique of~$G$.
The graph~$G^K$ satisfies Condition~\eqref{eq:star} if and only if it satisfies Condition~\eqref{eq:natural}.
\end{lemma}

The main theorems of the paper are easy consequences of Lemmas~\ref{thm:star} and~\ref{lem:star-sharp-equivalence-in-C_4-free}.

\begin{proof}[Proof of Theorem~\ref{thm:main-1}]
Let $G$ be a $C_4$-free graph.
Suppose $G$ is a circular-arc graph. 
By Theorem~\ref{thm:correlation}, there exists a clique $K$ of $G$ such that $G^K$ is an interval graph
satisfying Condition~\eqref{eq:sharp}.
By Observation~\ref{obs:sharp-implies-star}, $G^K$ satisfies Condition~\eqref{eq:star}.
By Lemma~\ref{thm:star},
$G^K$ contains no configuration in Figure~\ref{fig:forbidden-configurations-C_4-free}.

For the proof of the other direction, let $K$ be a clique in $G$ such that~$G^K$ satisfies~\eqref{eq:natural}.
Hence, Lemma~\ref{thm:star} asserts that $G^K$ satisfies Condition~\eqref{eq:star}.
Since $G$ is $C_4$-free, Lemma~\ref{lem:star-sharp-equivalence-in-C_4-free} asserts that $G^K$ satisfies Condition~\eqref{eq:sharp}.
By Theorem~\ref{thm:correlation}, we conclude that $G$ is a circular-arc graph.
\end{proof}

\begin{proof}[Proof of Theorem~\ref{thm:main-chordal}]
Let $G$ be a chordal graph and let $s$ be a simplicial vertex in $G$.
If $G^{N[s]}$ is an interval graph containing no configuration shown in Figure~\ref{fig:forbidden-configurations-C_4-free}, 
$G$ is a circular-arc graph by Theorem~\ref{thm:main-1}.
On the other way, since in every normalized model of $G$ the arc of $s$ is covered by and only by the arcs representing
the vertices from $N_G[s]$ \cite[Lemma 2.2]{cao-24-split-cag}, we conclude that~$G^{N[s]}$ is an interval graph containing no configuration in Figure~\ref{fig:forbidden-configurations-C_4-free}.
\end{proof}

For~$k \ge 3$, the \emph{$k$-sun}, denoted as~$S_{k}$, is the graph obtained from the cycle of length~${2 k}$ by adding all edges among the even-numbered vertices to make them a clique.  For example, $S_{3}$ is depicted in Figures~\ref{fig:normal-and-helly}a.
Joeris et al.~\cite{joeris-11-hcag} proved that a~$C_{4}$-free circular-arc graph is a Helly circular-arc graph if and only if it does not contain an induced copy of the complement of any $k$-sun, $k \ge 3$.
Note that in~$\overline{S_{k}}$, the degree of a vertex is either~$k - 2$ or~$2 k - 3$.

\begin{theorem}[\cite{joeris-11-hcag}]
  \label{thm:joeris}
  Let~$G$ be a minimal circular-arc graph that is not a Helly circular-arc graph.  If~$G$ is~$C_{4}$-free, then it is~$\overline{S_{k}}$.
\end{theorem}

\begin{proof}[Proof of Theorem~\ref{thm:main-helly}]
  The necessity follows from Lemma~\ref{lem:construction}.  In any Helly circular-arc model of~$G$, the arcs for vertices in~$K$ share a point, and this point cannot be in any other arc because~$K$ is maximal.
  
  We now consider the sufficiency.
  By Theorem~\ref{thm:main-1},~$G$ is a circular-arc graph.
  If it is not a Helly circular-arc graph, then it contains an induced~$\overline{S_{k}}$ by Theorem~\ref{thm:joeris}.  However,~$G^{K}$ contains a hole if we take~$K$ to be the vertices of degree~$2 k - 3$.
\end{proof}

\section{Proof of Lemma~\ref{thm:star}}

Since the necessity of Lemma~\ref{thm:star} is straightforward, we will be focused on the sufficiency; i.e., if $G^K$ contains no configuration in Figure~\ref{fig:forbidden-configurations-C_4-free}, then $G^K$ satisfies Condition~\eqref{eq:star}.

Booth and Lueker~\cite{BoothLueker76} introduced a data structure PQ-tree to represent all clique paths of the interval graph~$H$.
A~\emph{PQ-tree} $\pqtree$ of the interval graph $H$ is a rooted and labeled tree in which
\begin{itemize}
\item the maximal cliques of $H$ are in the correspondence with the leaf nodes of $\pqtree$;
\item each non-leaf node is labeled either ``P'' or ``Q'';
\item a~P-node has at least two children, and its non-leaf children must be Q-nodes; and
\item a~Q-node has at least three children, and they are ordered.
\end{itemize}
See Figure~\ref{fig:PQ-tree-nor-int-mod} for an illustration.
The main observation of Booth and Lueker~\cite{BoothLueker76} is the following.
\begin{theorem}[\cite{BoothLueker76}]
\label{thm:pq-trees}
Let~$H$ be an interval graph and~$T$ a PQ-tree of~$H$.
An ordering of the maximal cliques of~$H$ is a clique path of~$H$ if and only if it is the linear order of the leaves of a tree obtained from~$\pqtree$ by permuting the children of P-nodes and reversing the children of some Q-nodes.
\end{theorem}

\begin{figure}[h]
\begin{subfigure}[t]{0.6\linewidth}
    \begin{tikzpicture}[scale=1.5]\small
      \draw[region, fill=Cyan!30] (.6, 0.6) rectangle (2.2, 1.4);
      \draw[region, fill=Cyan!30] (2.4, 0.6) rectangle (6.2, 2.);
      \draw[region, fill=violet!30] (2.5, .9) rectangle (5.2, 1.8);
      \begin{scope}[every path/.style={{|[right]}-{|[left]}}]
        \foreach \x in {1, ..., 6} \node[gray] at (\x, -.15) {$K_\x$};
        \foreach[count=\i] \l/\r/\y in {1/6/2, 1/6/1, 1/2/4, .98/1.02/6/1, .98/1.02/5/1, 1.98/2.02/5, 3/6/4, 3/5/5, 3/4/6, 2.98/3.02/7, 4.98/5.02/6, 4/5/8, 4/5/7, 5.98/6.02/7, 5.98/6.02/6, 5.98/6.02/5}{
          \draw[thick] (\l-.02, \y/5) node[left] {$\i$} to (\r+.02, \y/5);
        }
      \end{scope}
      \draw[dashed, gray] (.8, 0) -- (6.2, 0);
    \end{tikzpicture}

\caption{}\label{fig:normalizedModelx}
\end{subfigure}
\,
\begin{subfigure}[t]{0.3\linewidth}
  \centering
\begin{tikzpicture}[
  every node/.style={circle, draw, minimum size=15pt, inner sep=.5pt},
  P/.style={circle, draw,fill=cyan!30},
  Q/.style={circle, draw,fill=violet!30},
  level 1/.style={level distance=5ex, sibling distance=7em},
  level 2/.style={sibling distance=4em},
  level 3/.style={sibling distance=2.5em}
  ]
  \small
      \node[P]{P}
      [edge from parent]child{node[P] {P}
      child {node{$K_{1}$}}
      child {node{$K_{2}$}}
    }
    child{node[P] {P}
      child {node[Q]{Q}
        child {node{$K_{3}$}}
        child {node{$K_{4}$}}
        child {node{$K_{5}$}}
      }
      child {node {$K_{6}$}
    }
    }
    ;
\end{tikzpicture}
\caption{}\end{subfigure}
\caption{(a) An interval graph given as a clique path, and (b) a PQ-tree $\pqtree$.}
\label{fig:PQ-tree-nor-int-mod} 
\end{figure}
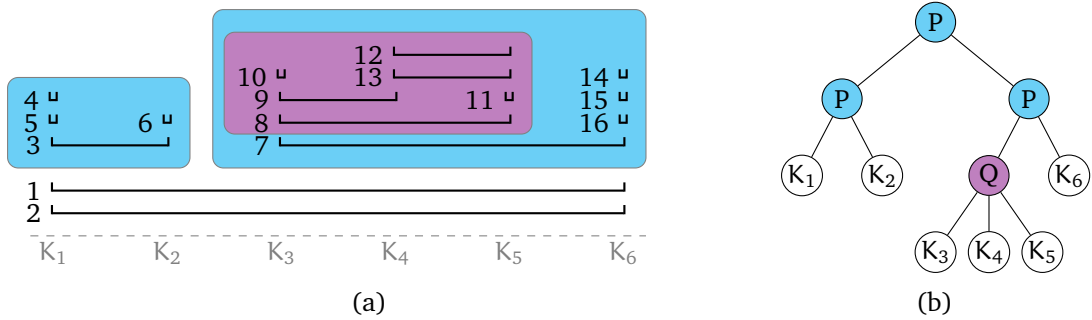

For every vertex~$v$ of~$H$, we denote by $\cliques(v)$ the set of all maximal cliques of $H$ containing the vertex $v$.
For every node $A$ in $\pqtree$ we denote by $\cliques(A)$ the set of all leaf cliques in the subtree rooted at~$A$. 
We abuse the notation by using~$H[A]$ to denote the graph induced by the set 
$V(A) = \bigcup \cliques(A)$ of vertices that appear in the descendants of~$A$.
Let
\[
  \inner(A) = \{v \in V(H)\mid \cliques(v) \subseteq \cliques(A)\}, \quad \encomp(A) = V(A)\setminus \inner(A).
\]
Note that a vertex~$v$ is in~$\encomp(A)$ if and only if~$\cliques(A) \subsetneq \cliques(v)$, and~$v$ is universal in $H[A]$ if and only if~$\cliques(A) \subseteq \cliques(v)$.\footnote{A reader familiar with modular decomposition may have noticed that~$\inner(A)$ is a strong module of~$H$. Indeed, $\inner(A)$ is a prime module if~$A$ is a Q-node; otherwise, it is a series module when it has universal vertices, or a parallel node otherwise}
For example, in the interval graph~$H$ from Figure~\ref{fig:PQ-tree-nor-int-mod},
we have $\inner(Q) = \{8,9,10,11,12,13\}$ and $\encomp(Q) = \{1,2,7\}$. 
The following are straightforward from the definition of PQ-trees.
We present proofs in the appendix for the sake of completeness.

\begin{proposition}\label{prop:PQ-tree-nodes}
  Let~$H$ be an interval graph and~$T$ a PQ-tree of~$H$.
  Let~$A$ be a non-leaf node in~$T$ and let~$\univ(A)$ be the set of universal vertices in~$H[A]$.

  \begin{enumerate}[1)]
  \item \label{item:PQ-tree-ex-univ} $\encomp(A)\subseteq \univ(A)$.
  \item \label{item:PQ-tree-connectivity} The subgraph~$H[A] - \univ(A)$ is connected if and only if~$A$ is a Q-node.
  \item   If~$A$ is a P-node, then for each child~$B$ of~$A$, the set $\inner(B)$ is a component in $H[A] - \univ(A)$.
  \end{enumerate}
\end{proposition}
\begin{proposition}\label{prop:V-sets}
  Let~$H$ be an interval graph and~$T$ a PQ-tree of~$H$.
  Let~$A$ be a Q-node in~$T$ and let $\langle B_1,\ldots,B_\ell \rangle$ be its children in order.
  For every pair~$i$ and~$j$ with~$1\le i < j < \ell$ or~$1< i < j \le \ell$,
  \[
    V_{i-1} \cap V_{i} \setminus V_j \neq \emptyset \quad \text{or}  \quad V_j \cap V_{j+1} \setminus V_{i} \neq \emptyset,
  \]
  where $V_i = V(B_i)$ for $i=1,\ldots,\ell$, and $V_0 = V_{\ell+1} = \emptyset$.
\end{proposition}

Now, we return to the proof of Lemma~\ref{thm:star}.
Let $\pqtree$ be a PQ-tree of the graph~$G^K$ and let~$Z$ be the root node of~$\pqtree$.
We process the PQ-tree $\pqtree$ in a bottom-up manner, and for every node $A$ in $\pqtree$ we construct a clique path $\langle K_1,\ldots,K_{\ell} \rangle$ for $G^K[A]$ that satisfies 
the following conditions:
\begin{subequations}
  \begin{alignat}{6}
    (v,u) \in \mathcal{P} &\cap \inner(A)\times \inner(A) &\implies& \lk(u) = \lk(v) &\lor& \rk(v) = \rk(u),
                                                                                                     \label{eq:main_cond_inner_vertices} 
    \\
    (v,u) \in \mathcal{P} &\cap \encomp(A)\times \inner(A) &\implies& \lk(u)=1 &\lor& \rk(u)=\ell.
                                                                                               \label{eq:main_cond_extreme_vertices} 
  \end{alignat}
\end{subequations}
If both~\eqref{eq:main_cond_inner_vertices} and~\eqref{eq:main_cond_extreme_vertices} hold, we say the clique path~$\langle K_1,\ldots,K_{\ell} \rangle$ of~$G^K[A]$ is \emph{admissible} for the node~$A$.
Note that for any clique path of~$G^K$ satisfying Condition~\eqref{eq:star}, its restriction to the set~$\cliques(A)$ is admissible for $A$, for any node $A$ of $T$.
Since~$\inner(Z) = V(G^K)$,
that $G^{K}$ satisfies Condition~\eqref{eq:star} is equivalent to that~$Z$ satisfies \eqref{eq:main_cond_inner_vertices}.
Condition~\eqref{eq:main_cond_extreme_vertices}, which is satisfied for~$Z$ vacuously because~$\encomp(Z) = \emptyset$, is introduced for the convenience of the inductive construction.
For the induction, we will be focused on the set of vertices that must be put in the end cliques of clique paths of~$G^K[A]$:
\[
  \extreme(A) = \big{\{} u \in \inner(A)\setminus K\mid \text{ there exists } v \in \encomp(A)\cap K \text{ such that } (u,v) \in \mathcal{P}  \big{\}}.
\]
Condition~\eqref{eq:main_cond_extreme_vertices} is then equivalent to that 
an admissible clique path of $A$ must contain every vertex from $\extreme(A)$ in its first or its last clique.

It is simpler to deal with P-nodes.

\begin{lemma}\label{lem:p-node-induction}
Let $A$ be a P-node.
We can find an admissible clique path for~$G^{K}[A]$ if
  \begin{itemize}
  \item $G^{K}$ satisfies Condition~\eqref{eq:natural}, and
  \item there is an admissible clique path for~$G^{K}[B]$ for all children~$B$ of~$A$.
  \end{itemize} 
\end{lemma}
\begin{proof}
  By definition, all the non-leaf children of~$A$ are Q-nodes.
  For each child~$B$ of~$A$, we fix an admissible clique path for~$G^{K}[B]$.

  \begin{figure}[h]
  \centering\small
  \begin{subfigure}[t]{0.8\linewidth}
    \centering
    \begin{tikzpicture}[xscale=0.95,yscale=0.7,>=latex,shorten >=-0.4pt,shorten <=-0.4pt]
\foreach \i in {1, 2, 3} {
\draw[region](\i*2-2., 0.8) rectangle (\i*2-.5, 2);
        \node at (\i*2- 1.25, 0.3) {$B_{\i}$};
}
      \draw[{|[right]}-{|[left]}] (0, 1.2) node[above right] {$u_{1}$} -- (.5, 1.2);
      \draw[{|[right]}-{|[left]}] (2, 1.2) node[above right] {$u_{2}$} -- (3.5, 1.2);
      \draw[{|[right]}-{|[left]}] (4.8, 1.2) node[above right] {$u_{3}$} -- (5.5, 1.2);
      
      \foreach \i in {1, 2, 3} {
        \draw[-] (0, 3.5 - \i/3) -- (5.5, 3.5 - \i/3);
        \draw[dotted] (0, 3.5 - \i/3) -- ++(-.3, 0) node[left] {$v_{\i}$};
        \draw[dotted] (5.5, 3.5 - \i/3) -- ++(.3, 0);
      }

      \draw[white] (6.5,0) -- (6.5,1) {};
    \end{tikzpicture}
    \caption{}
    \label{fig:P-node-case-e}
  \end{subfigure}

  \begin{subfigure}[t]{0.2\linewidth}
    \centering
    \begin{tikzpicture}[scale=.9]
      \node[a-vertex, "$v_1$"] (v) at (2, 1) {};
      \foreach \i in {1, 2, 3}
        \node[b-vertex, "$u_\i$" below] (u\i) at (\i, 0) {};
      \foreach \i in {1, 2, 3}
        \draw[very thick, teal] (v) -- (u\i);
    \end{tikzpicture}
    \caption{}
    \label{fig:fc-P-node-case-g}
  \end{subfigure}
\begin{subfigure}[t]{0.2\linewidth}
    \centering
    \begin{tikzpicture}[scale=.9]
      \foreach \i in {1, 2}
      \node[a-vertex, "$v_\i$"] (v\i) at (.5+\i, 1) {};
      \foreach \i in {1, 2, 3} {
        \node[b-vertex, "$u_\i$" below] (u\i) at (\i, 0) {};
      }
      \draw (v1) -- (v2);
      \draw (v1) -- (u3) (v2) -- (u1);
      \draw[witnessed edge] (v1) -- (u1) (v2) -- (u2) (v2) -- (u3);
      \uncertain{v1}{u2};
    \end{tikzpicture}
    \caption{}
    \label{fig:fc-P-node-case-f}
  \end{subfigure}
  \begin{subfigure}[t]{0.2\linewidth}
    \centering
    \begin{tikzpicture}[scale=.9]
      \foreach \i in {1, 2, 3}
      \node[a-vertex, "$v_\i$"] (v\i) at (\i, 1) {};
      \foreach \i in {1, 2, 3}
        \node[b-vertex, "$u_\i$" below] (u\i) at (\i, 0) {};
      \foreach \i in {1, 2, 3}
      \foreach \j in {1, 2, 3} {
        \ifnum \i=\j
          \def\c{witnessed edge}\else
          \def\c{}\fi
        \draw[\c] (v\i) -- (u\j);
        }
      \draw (v1) -- (v2) -- (v3);
      \draw[bend left=20] (v1) edge (v3);
    \end{tikzpicture}
    \caption{}
    \label{fig:fc-P-node-case-e}
  \end{subfigure}

  \caption{Forbidden configurations encountered when
    Claim~\eqref{eq:P-node-three-children} is violated.  Vertices~$v_{1}$, $v_{2}$, and~$v_{3}$ might be the same.
    The intervals for~$v_{i}, i = 1, 2, 3$, extend from at least one side, indicated by dotted lines.
}
  \label{fig:P-node-three-children} 
\end{figure}
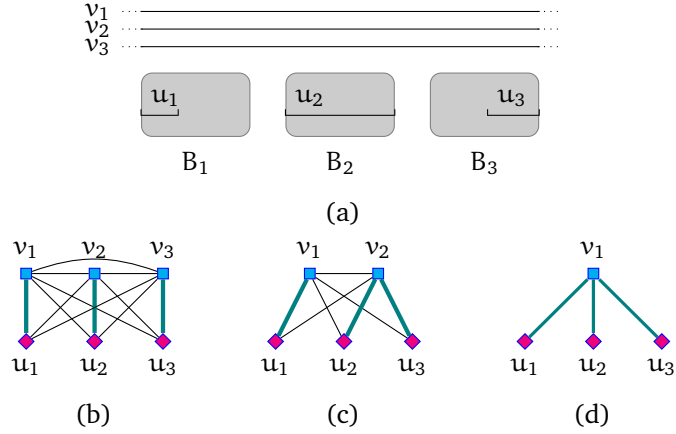

  First, we claim that: 
\begin{equation}
\label{eq:P-node-three-children}
\text{the node~$A$ has at most two children $B$ such that $\extreme(B) \neq \emptyset$.}
\end{equation}
Suppose for contradiction that~$A$ has three distinct children~$B_1$, $B_2$, and~$B_3$ such that~$\extreme(B_i) \neq \emptyset$ for~$i=1,2,3$ (see Figure~\ref{fig:P-node-three-children} for an illustration).
In other words, for~$i=1,2,3$, there are~$u_i \in \inner(B_i)$ and~$v_i \in \encomp(B_i)$ such that~$(u_i, v_i) \in \mathcal{P}$.
Note that $v_1,v_2$, and~$v_3$ are not necessarily distinct.
By Proposition~\ref{prop:PQ-tree-nodes}\eqref{item:PQ-tree-ex-univ}, all the edges between~$\{v_1,v_2,v_3\}$ and~$\{u_1,u_2,u_3\}$ are present in~$G^{K}$.
If $v_1$, $v_2$, and $v_3$ coincide, then $G^{K}[\{v_1,u_1,u_2,u_3\}]$ is Configuration~\ref{fig:claw-unlabeled} (reproduced and labeled in Figure~\ref{fig:fc-P-node-case-g}).
Hence, assume without loss of generality that $v_2 = v_3 \neq v_1$.  We are in the previous case if~$(u_{1}, v_{2}) \in \mathcal{P}$ or if~$(u_{2}, v_{1}), (u_{3}, v_{1}) \in \mathcal{P}$.
Thus,~$(u_{1}, v_{2}) \not\in \mathcal{P}$, and we may renumber~$u_{2}$ and~$u_{3}$ such that~$(u_{3}, v_{1}) \not\in \mathcal{P}$.
Then $G^{K}[\{v_1,v_2,u_1,u_2,u_3\}]$ is Configuration~\ref{fig:double-claw-unlabeled} (reproduced and labeled in Figure~\ref{fig:fc-P-node-case-f}).
In the rest,~$v_1$, $v_2$, and $v_3$ are all distinct.
If there are distinct~$i, j\in\{1, 2, 3\}$ such that~$(u_{i}, v_{j}) \in \mathcal{P}$, then we are in the previous case.
Otherwise,~$G^{K}[\{v_1,v_2,v_3,u_1,u_2,u_3\}]$ is Configuration~\ref{fig:triple-claw-unlabeled} (reproduced and labeled in Figure~\ref{fig:fc-P-node-case-e}).
  This concludes Claim~\eqref{eq:P-node-three-children}.

\begin{figure}[h]
  \centering \small
    \begin{subfigure}[t]{0.45\linewidth}
    \centering
    \begin{tikzpicture}[xscale=1.25,yscale=0.7,>=latex,shorten >=-0.4pt,shorten <=-0.4pt]
      \foreach \i in {1, 2} {
\draw[region](\i*2.5 - 2.5, 0.5) rectangle (\i*2.5 + .5 - \i, 2);
        \draw[{|[right]}-{|[left]}] (\i*1.5 - 1.5, 1.3) node[above right] {$u_{\i}$} -- ({\i*1.5- 1}, 1.3);
      }
      \draw[{|[right]}-{|[left]}] (2.5, 1.3) node[above right] {$x_{0}$} -- (3.5, 1.3);
      \draw[dashed, thick] (0.3,1.1) -- (1.8, 1.1);
      \node at (1, 0.2) {$B$};
      
      \foreach \i in {1, 2} {
        \draw[-] (0, 3. - \i/3) -- (3.5, 3. - \i/3);
        \draw[dotted] (0, 3. - \i/3) -- ++(-.3, 0) node[left] {$v_{\i}$};
        \draw[dotted] (3.5, 3. - \i/3) -- ++(.3, 0);
      }

      \draw[white] (6.5,0) -- (6.5,1) {};
    \end{tikzpicture}
    \caption{}
    \label{fig:P-node-case-a}
  \end{subfigure}
  \begin{subfigure}[t]{0.45\linewidth}
    \begin{tikzpicture}[xscale=1.25,yscale=0.7,>=latex,shorten >=-0.4pt,shorten <=-0.4pt, label distance=-4pt]
      \foreach \i/\p in {1/above, 2/below} {
\draw[region](\i*2.5 - 2.5, 0.5) rectangle (\i*2.5 + .5 - \i, 2);
        \pgfmathsetmacro{\y}{1.2 + Mod(\i, 2)/5}        
        \draw[{|[right]}-{|[left]}] (\i*1.5 - 1.5, \y) -- node[midway, "$x_{\i}$" \p] {} ({\i*1.5- 1}, \y);
\draw[{|[left]}-{|[right]}] (2.7 - \i*.7, \y) -- node[midway, "$u_{\the\numexpr3-\i\relax}$" \p] {} (1.4 - \i*.7, \y);
}
      \draw[{|[right]}-{|[left]}] (2.5, 1.3) node[above right] {$x_{0}$} -- (3.5, 1.3);
      \node at (1, 0.2) {$B$};
      
      \foreach \i in {1, 2} {
        \draw[-] (0, 3. - \i/3) -- (3.5, 3. - \i/3);
        \draw[dotted] (0, 3. - \i/3) -- ++(-.3, 0) node[left] {$v_{\i}$};
        \draw[dotted] (3.5, 3. - \i/3) -- ++(.3, 0);
      }

      \draw[white] (6.5,0) -- (6.5,1) {};
    \end{tikzpicture}
    \caption{}
    \label{fig:P-node-case-c}
  \end{subfigure}

  \begin{subfigure}[t]{0.22\linewidth}
    \centering
    \begin{tikzpicture}[scale=.72,yscale=1.25]
      \node[empty vertex] (x1) at (2.0,0) {};
      \coordinate (lx1) at (2,-0.3) {};
      \node[empty vertex] (xp) at (3.0,0) {};
      \coordinate (lxp) at (3,-0.3) {};
      \node[b-vertex] (u1) at (1, 0) {};
      \coordinate (lu1) at (1,-0.3) {};
      \node[b-vertex] (u2) at (4, 0) {};
      \coordinate (lu2) at (4,-0.3) {};
      \node[a-vertex] (v) at (2.5, 1) {};
      \coordinate (lv) at (2.9,1) {};
      \node[empty vertex] (x0) at (2.5,1.75) {};
      \coordinate (lx0) at (2.9,1.75) {};
      \draw (v) -- (x0)  {};
      \draw[ultra thick, teal] (v) -- (u1);
      \draw[ultra thick, teal] (v) -- (u2);
      \draw (v) -- (x1) {};
      \draw (v) -- (xp) {};
      \draw (u1) -- (x1) {};
      \draw[dashed] (x1) -- (xp) {};
      \draw (xp) -- (u2) {};
      \tikzstyle{every node}=[inner sep=1pt]
      \begin{footnotesize}
        \node at (lv) {$v_1$};
        \node at (lu1) {$u_1$};
        \node at (lu2) {$u_2$};
        \node at (lx0) {$x_0$};
        \node at (lx1) {$x_1$};
        \node at (lxp) {$x_p$};
      \end{footnotesize}
    \end{tikzpicture}
    \caption{}
    \label{fig:fc-P-node-case-b}
  \end{subfigure}
  \begin{subfigure}[t]{0.22\linewidth}
    \centering
    \begin{tikzpicture}[scale=.72,yscale=1.25]
      \node[empty vertex] (x1) at (2.0,0) {};
      \coordinate (lx1) at (2,-0.3) {};
      \node[empty vertex] (xp) at (3.0,0) {};
      \coordinate (lxp) at (3,-0.3) {};
      \node[b-vertex] (u1) at (1, 0) {};
      \coordinate (lu1) at (1,-0.3) {};
      \node[b-vertex] (u2) at (4, 0) {};
      \coordinate (lu2) at (4,-0.3) {};
      \node[a-vertex] (v1) at (2, 1) {};
      \coordinate (lv1) at (1.6,1) {};
      \node[a-vertex] (v2) at (3, 1) {};
      \coordinate (lv2) at (3.4,1) {};
      \node[empty vertex] (x0) at (2.5,1.75) {};
      \coordinate (lx0) at (2.9,1.75) {};
      \draw (v1) -- (x0)  {};
      \draw (v2) -- (x0)  {};
      \draw (v1) -- (v2)  {};
      \draw[ultra thick, teal] (v1) -- (u1);
      \draw[ultra thick, teal] (v2) -- (u2);
      \draw (v1) -- (x1) {};
      \draw (v1) -- (xp) {};
      \draw (v1) -- (u2) {};
      \draw (v2) -- (u1) {};
      \draw (v2) -- (x1) {};
      \draw (v2) -- (xp) {};
      \draw (u1) -- (x1) {};
      \draw[dashed] (x1) -- (xp) {};
      \draw (xp) -- (u2) {};
      \tikzstyle{every node}=[inner sep=1pt]
      \begin{footnotesize}
        \node at (lv1) {$v_1$};
        \node at (lv2) {$v_2$};
        \node at (lu1) {$u_1$};
        \node at (lu2) {$u_2$};
        \node at (lx0) {$x_0$};
        \node at (lx1) {$x_1$};
        \node at (lxp) {$x_p$};
      \end{footnotesize}
    \end{tikzpicture}
    \caption{}
    \label{fig:fc-P-node-case-a}
  \end{subfigure}
  \hspace{0.25cm}
  \begin{subfigure}[t]{0.22\linewidth}
    \centering
    \begin{tikzpicture}[scale=.72,yscale=1.25]
      \node[empty vertex] (x1) at (1.0,0) {};
      \coordinate (lx1) at (1,-0.4) {};
      \node[empty vertex] (x2) at (4.0,0) {};
      \coordinate (lx2) at (4,-0.4) {};
      \node[b-vertex] (u1) at (2, 0) {};
      \coordinate (lu1) at (2,-0.4) {};
      \node[b-vertex] (u2) at (3, 0) {};
      \coordinate (lu2) at (3,-0.4) {};
      \node[a-vertex] (v) at (2.5, 1) {};
      \coordinate (lv) at (2.9,1) {};
      \node[empty vertex] (x0) at (2.5,1.75) {};
      \coordinate (lx0) at (2.9,1.75) {};
      \draw (v) -- (x0)  {};
      \draw[ultra thick, teal] (v) -- (u1);
      \draw[ultra thick, teal] (v) -- (u2);
      \draw (v) -- (x1) {};
      \draw (v) -- (x2) {};
      \draw (x1) -- (u1) -- (u2) --(x2) {};
      \tikzstyle{every node}=[inner sep=1pt]
      \begin{footnotesize}
        \node at (lv) {$v_1$};
        \node at (lu1) {$u_1$};
        \node at (lu2) {$u_2$};
        \node at (lx0) {$x_0$};
        \node at (lx1) {$x_1$};
        \node at (lx2) {$x_2$};
      \end{footnotesize}
    \end{tikzpicture}
    \caption{}
    \label{fig:fc-P-node-case-d}
  \end{subfigure}
  \begin{subfigure}[t]{0.22\linewidth}
    \centering
    \begin{tikzpicture}[scale=.72,yscale=1.25]
      \node[empty vertex] (x1) at (1.0,0) {};
      \coordinate (lx1) at (1,-0.4) {};
      \node[empty vertex] (x2) at (4.0,0) {};
      \coordinate (lx2) at (4,-0.4) {};
      \node[b-vertex] (u1) at (2, 0) {};
      \coordinate (lu1) at (2,-0.4) {};
      \node[b-vertex] (u2) at (3, 0) {};
      \coordinate (lu2) at (3,-0.4) {};
      \node[a-vertex] (v1) at (2, 1) {};
      \coordinate (lv1) at (1.6,1) {};
      \node[a-vertex] (v2) at (3, 1) {};
      \coordinate (lv2) at (3.4,1) {};
      \node[empty vertex] (x0) at (2.5,1.75) {};
      \coordinate (lx0) at (2.9,1.75) {};
      \draw (v1) -- (x0)  {};
      \draw (v1) -- (v2)  {};
      \draw (v2) -- (x0)  {};
      \draw[ultra thick, teal] (v1) -- (u1);
      \draw[ultra thick, teal] (v2) -- (u2);
      \draw (v1) -- (x1) {};
      \draw (v1) -- (x2) {};
      \draw (v1) -- (u2) {};
      \draw (v2) -- (u1) {};
      \draw (v2) -- (x1) {};
      \draw (v2) -- (x2) {};
      \draw (x1) -- (u1) -- (u2) --(x2) {};
      \tikzstyle{every node}=[inner sep=1pt]
      \begin{footnotesize}
        \node at (lv1) {$v_1$};
        \node at (lv2) {$v_2$};
        \node at (lu1) {$u_1$};
        \node at (lu2) {$u_2$};
        \node at (lx0) {$x_0$};
        \node at (lx1) {$x_1$};
        \node at (lx2) {$x_2$};
      \end{footnotesize}
    \end{tikzpicture}
    \caption{}
    \label{fig:fc-P-node-case-c}
  \end{subfigure}
  \caption{Forbidden configurations encountered when
    Claim~\eqref{eq:P-node-one-child} is violated.  Vertices~$v_{1}$ and~$v_{2}$ might be the same (c, e).
    The intervals for~$v_{i}, i = 1, 2$, extend from at least one side, indicated by dotted lines.
}
  \label{fig:P-node-one-child} 
\end{figure}

Next, we claim that for each child~$B$ of~$A$, 
  \begin{equation}
    \label{eq:P-node-one-child}
    \text{One end of the clique path of~$G^{K}[B]$ contains
      all the vertices in $\extreme(B)$.}
\end{equation}
Claim~\ref{eq:P-node-one-child} holds vacuously when~$\extreme(B)=\emptyset$, and hence we assume otherwise.
Suppose for contradiction that there are vertices~$u_1$ and~$u_2$ in different ends (see Figure~\ref{fig:P-node-one-child} for an illustration).
Let
$v_1,v_2 \in \encomp(B)$ be such that $(u_1,v_1) \in \mathcal{P}$ and $(u_2,v_2) \in \mathcal{P}$.
It is possible that $v_1 =v_2$.
Let $B'$ be a child of $A$ different from $B$, and let $x_0$ be a vertex from $\inner(B')$.

Suppose first that $u_1u_2 \notin E(G^K)$.
Since $B$ is a Q-node, by Proposition~\ref{prop:PQ-tree-nodes}\eqref{item:PQ-tree-connectivity}, there is an induced path joining $u_1$ and $u_2$ with all inner vertices in $\inner(B)$, say $u_1x_1\ldots x_pu_2$.
By Proposition~\ref{prop:PQ-tree-nodes}\eqref{item:PQ-tree-ex-univ}, \[
  \inner(B)\cup \inner(B')\subseteq N_{G^{K}}(v_{1}) \cap N_{G^{K}}(v_{2}).
\]
If $v_1 = v_2$, then $G^K[\{v_1,x_0,u_1,x_1,\ldots,x_p,u_2\}]$ is Configuration~\ref{fig:dag+2e-unlabeled} (reproduced and labeled in Figure~\ref{fig:fc-P-node-case-b}).
Hence, $v_1 \neq v_2$.
We are in the previous case if~$(u_{1}, v_{2})$ or~$(u_{2}, v_{1})$ is in~$\mathcal{P}$.
Otherwise,~$G^K[\{v_1,v_2,x_0,u_1,x_1,\ldots,x_p,u_2\}]$ is Configuration~\ref{fig:ddag+2e-unlabeled} (reproduced and labeled in Figure~\ref{fig:fc-P-node-case-a}).

Suppose then $u_1u_2 \in E(G^K)$.
Since both ends of the admissible clique path for $B$ are maximal cliques, we can find $x_1,x_2 \in \inner(B)$ such that $x_1u_1u_2x_2$ is an induced path in $G^K$. 
If $v_1 = v_2$, then $G^K[\{v_1,x_0,x_1,u_1,u_2,x_2\}]$ is Configuration~\ref{fig:(p4+p1)*1-unlabeled} (reproduced and labeled in Figure~\ref{fig:fc-P-node-case-d}).
Hence, $v_1 \neq v_2$.
We are in the previous case if~$(u_{1}, v_{2})$ or~$(u_{2}, v_{1})$ is in~$\mathcal{P}$.
Otherwise,~$G^K[\{x_0,v_1,v_2,x_1,u_1,u_2,x_2\}]$ is Configuration~\ref{fig:add-2-unlabeled} (reproduced and labeled in Figure~\ref{fig:fc-P-node-case-c}).
This verifies Claim~\eqref{eq:P-node-one-child}.

We now construct an admissible clique path for $G^K[A]$.
We put the at most two children~$B$ of $A$ with $\extreme(B) \neq \emptyset$ at the two ends.
For either of them, we make the end clique that contains all the vertices in $\extreme(B)$ at the end of the clique path for $G^K[A]$.
We put the clique paths of the remaining children in the middle.
Claims~\eqref{eq:P-node-three-children} and~\eqref{eq:P-node-one-child} assert that this clique path is admissible for $A$.
\end{proof}

For Q-nodes, we need the following result of Gimbel~\cite{gimbel-88-end-vertices}.
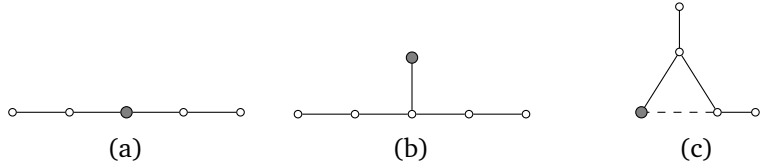
\begin{figure}[h]
  \centering \small
  \begin{subfigure}[b]{0.22\linewidth}
    \centering
    \begin{tikzpicture}[scale=.75]
      \node[filled vertex] (c) at (0, 0) {};
      \foreach \i in {1, 3} {
        \node[empty vertex] (u\i) at ({90*(\i-1)}:2) {};
        \node[empty vertex] (v\i) at ({90*(\i-1)}:1) {};
        \draw (u\i) -- (v\i) -- (c);
      }
    \end{tikzpicture}
    \caption{}
    \label{fig:gimble-5-path}
  \end{subfigure}
  \,
  \begin{subfigure}[b]{0.22\linewidth}
    \centering
    \begin{tikzpicture}[scale=.75]
      \node[empty vertex] (c) at (0, 0) {};
      \foreach \i in {1, 3} {
        \node[empty vertex] (u\i) at ({90*(\i-1)}:2) {};
        \node[empty vertex] (v\i) at ({90*(\i-1)}:1) {};
        \draw (u\i) -- (v\i) -- (c);
      }
      \draw (90:1) node[filled vertex] {} -- (c);
    \end{tikzpicture}
    \caption{}
    \label{fig:gimble-5-path-vertex}
  \end{subfigure}
  \,
  \begin{subfigure}[b]{0.22\linewidth}
    \centering
    \begin{tikzpicture}[yscale=.8]
      \draw[dashed] (0, 0) -- (1., 0);
      \draw (0.5, 1.75) node[empty vertex] {} -- (0.5, 1) node[empty vertex] (c) {} -- (0, 0) node[filled vertex] {};
      \foreach[count =\i] \x/\t in {1/empty } {
        \node[\t vertex] (v\i) at ({1. * \x}, 0) {};
        \draw (v\i) -- (c);
      }
      \draw (v1) -- (1.5, 0) node[empty vertex] {};
    \end{tikzpicture}
    \caption{}
    \label{fig:gimble-long-path}
  \end{subfigure}
  \caption{The solid node cannot be in an end clique in any clique path.  The dashed line in (c) indicates a path of a positive length.}
  \label{fig:gimbel-graphs}
\end{figure}
 
\begin{theorem}[\cite{gimbel-88-end-vertices}]
  \label{thm:end-interval}
  Let $H$ be an interval graph, and~$u$ a vertex of $H$.
There is a clique path of~$H$ in which $u$ is in the first or last clique
  if and only if $u$ is not the highlighted vertex of an induced subgraph in Figure~\ref{fig:gimbel-graphs}.
\end{theorem}

\begin{lemma}\label{lem:q-node-induction}
Let~$A$ be a Q-node.
We can find an admissible clique path for~$G^{K}[A]$ if
\begin{itemize}
\item $G^{K}$ satisfies Condition~\eqref{eq:natural}, and
\item there is an admissible clique path for~$G^{K}[B]$ for all children~$B$ of~$A$.
\end{itemize}
\end{lemma}
\begin{proof}
  Let~$\ell$ be the number of children of~$A$, and let~$\langle B_1, \ldots, B_{\ell} \rangle$ be its children in order.
For~$i=1,\ldots, \ell$, let~$s_{i} = |\mathcal{K}(B_{i})|$ and~$V_{i} = V(B_{i})$.
For notational convenience, we introduce two empty sets $V_0$ and~$V_{\ell+1}$.
Note that for $i=1,\ldots,\ell$, we have $V_{i} \setminus V_{i-1} \neq \emptyset$ and $V_{i} \setminus V_{i+1} \neq \emptyset$.
For each~$u\in V(A)$, let~$\lc(u)$ and~$\rc(u)$ denote the smallest and, respectively, largest index~$i$ such that~$u\in V_{i}$.
Note that they are well defined because~$V(A) = \bigcup_{i=1}^{\ell}V_{i}$.

We want to construct an admissible clique path for~$G^{K}[A]$.  We note that it suffices to prove three claims.
The first is about the positions of~$\lc(u)$ and~$\rc(u)$ with respect to~$\lc(v)$ and~$\rc(v)$.
\begin{enumerate}[(1)]
\item  \label{eq:uv_pair}
\text{For every pair $(v, u) \in \mathcal{P}$ with~$u\in\inner(A)$, it holds $\lc(u) = \lc(v)$ or $\rc(v) = \rc(u)$}.
\end{enumerate}
The next two claims involve two distinct pairs~$(v_1,u_1)$ and~$(v_2,u_2)$ in~$\mathcal{P}\cap \encomp(B_{t})\times \inner(B_{t})$ for some~$t = 1, \ldots, \ell$.
  Let~$s_t = |\mathcal{K}(B_{t})|$ and let~$\langle K_1,\ldots, K_{s_t} \rangle$ be the admissible clique path of $B_t$.
  Since $\langle K_1,\ldots, K_{s_t} \rangle$ is admissible for $B_t$, vertices~$u_1$ and $u_2$ are in the end cliques of $\langle K_1,\ldots, K_{s_t} \rangle$, and we have either $\lc(v_1)=t$ or $\rc(v_1)=t$ for $i=1,2$ by Claim~\eqref{eq:uv_pair}.
  \begin{enumerate}[(1)]
    \setcounter{enumi}{1} 
  \item\label{eq:B}
    If~$\lc(v_{1}) = \rc(v_{2}) = t$, then at least one of~$u_{1}$ and~$u_{2}$ is universal in~$G^K[B_{t}]$ or they are at different ends of $\langle K_1,\ldots, K_{s_t} \rangle$.
  \item \label{eq:A}
    If~$\lc(v_{1}) = \lc(v_{2}) = t$ or~$\rc(v_{1}) = \rc(v_{2}) = t$, then~$u_{1}$ and~$u_{2}$ are at the same end of $\langle K_1,\ldots, K_{s_t} \rangle$.
\end{enumerate}
Before proving the claims, we show how to construct the claimed admissible clique path for~$G^{K}[A]$ assuming their correctness.
It suffices to replace each~non-leaf node $B_t$ in $\langle B_1, \ldots, B_{\ell}\rangle$ with an admissible clique path $\langle K_1, \ldots, K_{s_t}\rangle$ for $B_t$ such that for every $(v,u) \in \mathcal{P} \cap \encomp(B_t) \times \inner(B_t)$ we have $u \in K_1$ if $\lc(v) = t$ and $u \in K_{s_t}$ if $\rc(v) = t$.
Claims~\eqref{eq:B} and~\eqref{eq:A} assert that such a clique path of $G^K[B_t]$ exists.
Eventually, Claim~\eqref{eq:uv_pair} asserts that the clique path of $G^K[A]$ we have obtained is admissible for $A$.

The rest of the proof is devoted to proving the claims.
\begin{proof}[Proof of Claim~\eqref{eq:uv_pair}]
  \renewcommand\qedsymbol{$\lrcorner$}
  Suppose for contradiction that there exists a pair $(v, u) \in \mathcal{P}$ violating Claim~\eqref{eq:uv_pair}.
  By the definition of~$\mathcal{P}$, it holds~$\lc(v) \le \lc(u) \le \rc(u) \le \lc(v)$.
  Thus,
  \[
    \lc(v) < \lc(u) \le \rc(u) < \lc(v).
  \]
Thus,~$v \not\in \inner(B_i)$ for all~$i$.
By Theorem~\ref{thm:pq-trees}, in any clique path of~$G^{K}$, the children  of~$A$ appear in the order of either~$\langle B_{1}, \ldots, B_{\ell}\rangle$ or its reversal.
Thus, $u$~cannot be in an end clique of any clique path of $G^K[A]$.
By Theorem~\ref{thm:end-interval}, $u$ is the highlighted vertex of an induced subgraph of $G^K[A]$ shown in Figure~\ref{fig:gimbel-graphs}
(reproduced and labeled in Figures~\ref{fig:l-gimbel-a}, \ref{fig:l-gimbel-b}, \ref{fig:l-gimbel-c}).

\begin{figure}[h]
  \centering \small
  \begin{subfigure}[b]{0.25\linewidth}
    \centering
    \begin{tikzpicture}[xscale=.75, yscale=0.8]
      \node[filled vertex] (u) at (0, 0) {};
      \coordinate (lu) at (0,-0.35) {};
      \node[empty vertex] (x0) at (-2, 0) {};
      \node[empty vertex] (x1) at (-1, 0) {};
      \node[empty vertex] (x3) at (1, 0) {};
      \node[empty vertex] (x4) at (2, 0) {};
      \draw (x0)--(x1)--(u)--(x3)--(x4);
    \tikzstyle{every node}=[inner sep=1pt]
    
    \begin{footnotesize}
    \node at (lu) {$u$};
    \end{footnotesize}
    \end{tikzpicture}
    \caption{}
    \label{fig:l-gimbel-a}
  \end{subfigure}
  \hspace{0.75cm}
  \begin{subfigure}[b]{0.25\linewidth}
    \centering
    \begin{tikzpicture}[xscale=.75, yscale=0.8]
      \node[filled vertex] (u) at (0, 1.25) {};
      \coordinate (lu) at (0.35,1.25) {};
      \node[empty vertex] (x0) at (-2, 0) {};
      \node[empty vertex] (x1) at (-1, 0) {};
      \node[empty vertex] (x2) at (0, 0) {};
      \node[empty vertex] (x3) at (1, 0) {};
      \node[empty vertex] (x4) at (2, 0) {};
      \draw (x0)--(x1)--(x2)--(x3)--(x4);
      \draw (x2) edge (u);
    \tikzstyle{every node}=[inner sep=1pt]
    
    \begin{footnotesize}
    \node at (lu) {$u$};
    \end{footnotesize}
    
    \draw[white] (-0.5,-0.5)--(0.5,-0.5);
    \end{tikzpicture}
    \caption{}
    \label{fig:l-gimbel-b}
  \end{subfigure}
  \hspace{0.75cm}
    \begin{subfigure}[b]{0.25\linewidth}
    \centering
    \begin{tikzpicture}[xscale=.75, yscale=0.8]
      \node[filled vertex] (u) at (1.2, 0) {};
      \coordinate (lu) at (1.2,-0.35) {};
      \node[empty vertex] (x0) at (0, 1.75) {};
      \node[empty vertex] (x1) at (0.6, .875) {};
      \node[empty vertex] (x2) at (0, 0) {};
      \node[empty vertex] (x3) at (-1.2, 0) {};
      \draw (x0)--(x1)--(x2)--(x3);
      \draw[dashed] (x2) edge (u);
      \draw (x1) -- (u);
    \tikzstyle{every node}=[inner sep=1pt]
    
    \begin{footnotesize}
    \node at (lu) {$u$};
    \end{footnotesize}
    \end{tikzpicture}
    \caption{}
    \label{fig:l-gimbel-c}
  \end{subfigure}

  \begin{subfigure}[b]{0.25\linewidth}
    \centering
    \begin{tikzpicture}[xscale=.75, yscale=0.9]
      \node[a-vertex] (v) at (0, 1) {};
      \coordinate (lv) at (0,1.35) {};
      \node[b-vertex] (u) at (0, 0) {};
      \coordinate (lu) at (0,-0.35) {};
      \node[empty vertex] (x0) at (-2, 0) {};
      \node[empty vertex] (x1) at (-1, 0) {};
      \node[empty vertex] (x3) at (1, 0) {};
      \node[empty vertex] (x4) at (2, 0) {};
      \draw (x0)--(x1)--(u)--(x3)--(x4);
      \draw (v)--(x0);
      \draw (v)--(x1);
      \draw (v)--(x3);
      \draw (v)--(x4);
      \draw[ultra thick, teal] (v) -- (u);
    \tikzstyle{every node}=[inner sep=1pt]    
    \begin{footnotesize}
    \node at (lu) {$u$};
    \node at (lv) {$v$};
    \end{footnotesize}
    \draw[white] (-2.5,-0.5)--(-2.5,0.2);
    \draw[white] (2.5,1.2)--(2.5,2.3);
     \end{tikzpicture}    
      \caption{}
    \label{fig:fs-gimbel-a}
\end{subfigure}
\hspace{0.75cm}
  \begin{subfigure}[b]{0.25\linewidth}
    \centering
    \begin{tikzpicture}[scale=.75]
      \draw (-2.,0) -- (2.,0);
      \node[a-vertex, "$v$"] (a) at (0, 1.) {};
      \node[b-vertex, "$u$"] (u) at (0.75, 1) {};
      \foreach[count=\j] \x in {-2, -1, 0, 1, 2} {
        \draw (a) -- (\x, 0) node[empty vertex] (x\j) {};
      }
      \draw (a) edge[witnessed edge] (u) (u) edge (x3);
      \node at (0, -.55) {};
    \end{tikzpicture}
      \caption{}
    \label{fig:fs-gimbel-b}
\end{subfigure}
\hspace{0.75cm}
\begin{subfigure}[b]{0.25\linewidth}
    \centering
    \begin{tikzpicture}[xscale=.75, yscale=0.9]
      \node[b-vertex] (u) at (1.2, 0) {};
      \coordinate (lu) at (1.2,-0.35) {};
      \node[a-vertex] (v) at (-0.6, 0.875) {};
      \coordinate (lv) at (-0.95,0.875) {};
      \node[empty vertex] (x0) at (0, 1.75) {};
      \node[empty vertex] (x1) at (0.6, .875) {};
      \node[empty vertex] (x2) at (0, 0) {};
      \node[empty vertex] (x3) at (-1.2, 0) {};
      \draw (x0)--(x1)--(x2)--(x3);
      \draw[dashed] (x2) edge (u);
      \draw (x1) -- (u);
      \draw (v) -- (x0);
      \draw (v) -- (x1);
      \draw (v) -- (x2);
      \draw (v) -- (x3);
      \draw [ultra thick, teal] (v) -- (u);
      
    \tikzstyle{every node}=[inner sep=1pt]
    
    \begin{footnotesize}
    \node at (lu) {$u$};
    \node at (lv) {$v$};
    \end{footnotesize}
     \end{tikzpicture}    
      \caption{}
    \label{fig:fs-gimbel-c}
\end{subfigure}
  \caption{Forbidden interval configurations obtained from Gimbel induced subgraphs.}
  \label{fig:gimbel-forbidden-configurations}
\end{figure}
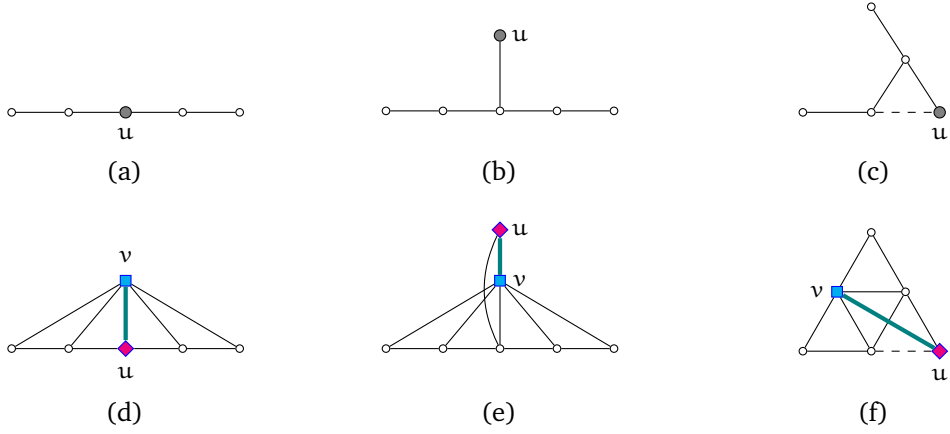

If~$v$ is universal in~$G^K[A]$ (i.e., $\lc(v) = 1$ and~$\rc(v) =  \ell$), then~$G^K$ contains Configuration~\ref{fig:p5x1-unlabeled} (reproduced and labeled in Figure~\ref{fig:fs-gimbel-a}) 
when the induced subgraph is Figure~\ref{fig:l-gimbel-a}, 
or the graph in Figure~\ref{fig:bent-whipping-top-unlabeled} (reproduced and labeled in Figure~\ref{fig:fs-gimbel-b}) when the induced subgraph is Figure~\ref{fig:l-gimbel-b},
or the graph in Figure~\ref{fig:ddag+e-unlabeled} (reproduced and labeled in Figure~\ref{fig:fs-gimbel-c}) when the induced subgraph is Figure~\ref{fig:l-gimbel-c}.
  
In the rest,~$v\not\in V_{1}\cap V_{\ell}$.
By Proposition~\ref{prop:V-sets}, at least one of 
\begin{align*}
  L =& V_{\lc(v) - 1}\cap V_{\lc(v)}\setminus V_{\rc(v)},
  \\
  R =& V_{\rc(v)}\cap V_{\rc(v) + 1}\setminus V_{\lc(v)}
\end{align*}
is nonempty.
In particular, $\lc(v) > 1$ if $L\ne\emptyset$, and~$\rc(v) < \ell$ if $R\ne\emptyset$.
Recall that, for $i = 1, \ldots, \ell$, neither $V_{i}\setminus V_{i+1}$ nor~$V_{i}\setminus V_{i-1}$ can be empty.
\begin{figure}[h]
  \centering \small
  \begin{subfigure}[b]{.3\linewidth}
    \centering
    \begin{tikzpicture}[xscale=.35]
      \foreach \l/\r/\y/\c in {4/5/2/u, 2/7/1/v}{
        \draw[{|[right]}-{|[left]}, thick, Maroon] (\l-.02, \y/4) to node[midway, yshift=3.5pt] {$\c$} (\r+.02, \y/4);
      }
      \foreach[count=\i] \right/\span/\y in {1/.5/1, 3/2.5/2}{
        \draw[-{|[left]}, thick] (\right-\span-.02, \y/4) node[left] {$x_\i$} to (\right+.02, \y/4);
      }
      \foreach[count=\i from 3] \left/\span/\y in {6/2.5/2, 8/.5/1}{
        \draw[{|[right]}-, thick] (\left-.02, \y/4) to (\left+\span+.02, \y/4) node[right] {$x_\i$};
      }
      \draw[gray] (.5, 0) -- (8.5, 0);
      \foreach \x in {1, ..., 8}
      \draw[dashed, gray] (\x, 0) -- ++(0, .1);
    \end{tikzpicture}
    \caption{}
    \label{fig:prime-Q-node-a}
  \end{subfigure}  
  \,
  \begin{subfigure}[b]{.3\linewidth}
    \centering
    \begin{tikzpicture}[xscale=.35]
      \foreach \l/\r/\y/\c in {4/6/2/u, 2/8/1/v}{
        \draw[{|[right]}-{|[left]}, thick, Maroon] (\l-.02, \y/4) to node[midway, yshift=3.5pt] {$\c$} (\r+.02, \y/4);
      }
\foreach[count=\i from 0] \l/\r/\y in {7.95/8.05/2}{
        \draw[{|[right]}-{|[left]}, thick] (\l-.02, \y/4) node[right] {$x_\i$} to (\r+.02, \y/4);
      }
      \foreach[count=\i] \right/\span/\y in {1/.5/1}{
        \draw[-{|[left]}, thick] (\right-\span-.02, \y/4) node[left] {$x_\i$} to (\right+.02, \y/4);
      }
      \foreach \l/\r/\y in {.7/4.5/3}{
        \draw[dashed] (\l-.02, \y/4) to (\r+.02, \y/4);
      }
      \draw[gray] (.5, 0) -- (8.5, 0);
      \foreach \x in {1, ..., 8}
      \draw[dashed, gray] (\x, 0) -- ++(0, .1);
    \end{tikzpicture}
    \caption{}
    \label{fig:prime-Q-node-b}
  \end{subfigure}  
  \,
  \begin{subfigure}[b]{.3\linewidth}
    \centering
    \begin{tikzpicture}[xscale=.3]
      \foreach \l/\r/\y/\c in {4/6/2/u, 2/8/1/v}{
        \draw[{|[right]}-{|[left]}, thick, Maroon] (\l-.02, \y/4) to node[midway, yshift=3.5pt] {$\c$} (\r+.02, \y/4);
      }
      \foreach[count=\i] \right/\span/\y in {1/.5/1, 2/1.5/2, 5/4.5/3} 
      \draw[-{|[left]}, thick] (\right-\span-.02, \y/4) node[left] {$x_\i$} to (\right+.02, \y/4);
      \foreach[count=\i from 4] \left/\span/\y in {8/0.5/2}
        \draw[{|[right]}-, thick] (\left-.02, \y/4) to (\left+\span+.02, \y/4) node[right] {$x_\i$};
      \draw[gray] (.5, 0) -- (8.5, 0);
      \foreach \x in {1, ..., 8}
      \draw[dashed, gray] (\x, 0) -- ++(0, .1);
    \end{tikzpicture}
    \caption{}
    \label{fig:prime-Q-node-c}
  \end{subfigure}  

  \begin{subfigure}[b]{.3\linewidth}
    \centering
    \begin{tikzpicture}[scale=.75]
      \node[a-vertex] (v) at (0, 0) {};
      \coordinate (lv) at (0,-0.3) {};
      \node[b-vertex] (u) at (0, 1) {};
      \coordinate (lu) at (0,1.3) {};
      \node[empty vertex] (x1) at (-2, 0) {};
      \coordinate (lx1) at (-2,-0.3) {};
      \node[empty vertex] (x2) at (-1, 0) {};
      \coordinate (lx2) at (-1,-0.3) {};
      \node[empty vertex] (x3) at (1, 0) {};
      \coordinate (lx3) at (1,-0.3) {};
      \node[empty vertex] (x4) at (2, 0) {};
      \coordinate (lx4) at (2,-0.3) {};
      
      \draw[ultra thick, teal] (v) -- (u) {};

      \draw (x1) -- (x2)--(v)--(x3)--(x4) {};
        \tikzstyle{every node}=[inner sep=1pt]
        \begin{footnotesize}
        \node at (lv) {$v$};
        \node at (lu) {$u$};
        \node at (lx1) {$x_1$};
        \node at (lx2) {$x_2$};
        \node at (lx3) {$x_3$};
        \node at (lx4) {$x_4$};
        \end{footnotesize}

     \end{tikzpicture}
    \caption{}
    \label{fig:fc-prime-Q-node-a}
  \end{subfigure}  
  \,
    \begin{subfigure}[b]{0.3\linewidth}
    \centering
    \begin{tikzpicture}[scale=.75]
      \node[a-vertex] (v) at (2, 1) {};
      \coordinate (lv) at (2.3,1) {};
      \node[empty vertex] (x0) at (2, 2) {};
      \coordinate (lx0) at (2.3,2) {};
      \node[empty vertex] (x1) at (0, 0) {};
      \coordinate (lx1) at (0,-0.3) {};
      \node[empty vertex] (x2) at (1, 0) {};
      \node[empty vertex] (x3) at (2, 0) {};
      \node[b-vertex] (u) at (3, 0) {};
      \coordinate (lu) at (3.3,0) {};

      \draw (x1) -- (x2)--(x3) {};
      \draw[dashed] (x3) -- (u) {};
      \draw[ultra thick, teal] (v) -- (u) {};
      \draw (v) -- (x0);
      \draw (v) -- (x2);
      \draw (v) -- (x3);
        \tikzstyle{every node}=[inner sep=1pt]
        \begin{footnotesize}
        \node at (lv) {$v$};
        \node at (lu) {$u$};
        \node at (lx0) {$x_0$};
        \node at (lx1) {$x_1$};
        \end{footnotesize}
    \end{tikzpicture}
    \caption{}
    \label{fig:fc-prime-Q-node-b}
  \end{subfigure}
  \,
    \begin{subfigure}[b]{0.3\linewidth}
    \centering
    \begin{tikzpicture}[scale=.75]
      \draw (1, 1) -- (4, 1);

       \node[empty vertex, "$x_{3}$" below] (v4) at (3, 0.2) {};
      \foreach[count=\i] \t/\v/\x in {empty /x_1/1, empty /x_2/2, a-/{\quad v}/3.5, b-/u/5} {
        \draw (v4) -- (\i, 1) node[\t vertex, "$\v$"] (u\i) {};
      }

      \draw (u3) -- ++(0, .85) node[empty vertex, "$x_{4}$"] (x2) {};
      \draw[witnessed edge] (u3) -- (u4);
     \end{tikzpicture}
    \caption{}
    \label{fig:fc-prime-Q-node-c}
  \end{subfigure}

\vspace{0.5cm}
  
  \begin{subfigure}[b]{.3\linewidth}
    \centering
    \begin{tikzpicture}[xscale=.35]
      \foreach \l/\r/\y/\c in {3/6/2/u, 1/8/1/v}{
        \draw[{|[right]}-{|[left]}, thick, Maroon] (\l-.02, \y/4) to node[midway, yshift=3.5pt] {$\c$} (\r+.02, \y/4);
      }

      \foreach[count=\i] \right/\span/\y in {2/.5/2, 4/3.5/3}{
        \draw[-{|[left]}, thick] (\right-\span-.02, \y/4) node[left] {$x_\i$} to (\right+.02, \y/4);
      }
      \foreach[count=\i from 3] \left/\span/\y in {6/2.5/3, 7/0.5/2}{
        \draw[{|[right]}-, thick] (\left-.02, \y/4) to (\left+\span+.02, \y/4) node[right] {$x_\i$};
      }
      \draw[gray] (.5, 0) -- (8.5, 0);
      \foreach \x in {1, ..., 8}
      \draw[dashed, gray] (\x, 0) -- ++(0, .1);

    \end{tikzpicture}
    \caption{}
    \label{fig:prime-Q-node-d}
  \end{subfigure}  
  \hspace{0.5cm}
  \begin{subfigure}[b]{.3\linewidth}
    \centering
    \begin{tikzpicture}[xscale=.3]
      \foreach \l/\r/\y/\c in {4/6/2/u, 2/8/1/v}{
        \draw[{|[right]}-{|[left]}, thick, Maroon] (\l-.02, \y/4) to node[midway, yshift=3.5pt] {$\c$} (\r+.02, \y/4);
      }
      \foreach[count=\i from 0] \right/\span/\y in {5/4.5/3, 1/.5/1} 
      \draw[-{|[left]}, thick] (\right-\span-.02, \y/4) node[left] {$x_\i$} to (\right+.02, \y/4);
      \foreach[count=\i from 4] \left/\span/\y in {8/.5/2}
        \draw[{|[right]}-, thick] (\left-.02, \y/4) to (\left+\span+.02, \y/4) node[right] {$x_\i$};

      \foreach[count=\i from 2] \l/\r/\y/\p in {2/3/2/left, 5.95/6.05/3/right}{
        \draw[{|[right]}-{|[left]}, thick] (\l-.02, \y/4) node[\p] {$x_\i$} to (\r+.02, \y/4);
      }

      \draw[gray] (.5, 0) -- (8.5, 0);
      \foreach \x in {1, ..., 8}
      \draw[dashed, gray] (\x, 0) -- ++(0, .1);
    \end{tikzpicture}
    \caption{}
    \label{fig:clique-path-5}
  \end{subfigure}  
  \hspace{0.5cm}
  \begin{subfigure}[b]{.3\linewidth}
    \centering
    \begin{tikzpicture}[xscale=.3]
      \foreach \l/\r/\y/\c in {2/4/2/u, 1/8/1/v}{
        \draw[{|[right]}-{|[left]}, thick, Maroon] (\l-.02, \y/4) to node[midway, yshift=3.5pt] {$\c$} (\r+.02, \y/4);
      }

      \foreach[count=\i] \l/\r/\y/\p in {0.95/1.05/2/left}{
        \draw[{|[right]}-{|[left]}, thick] (\l-.02, \y/4) node[\p] {$x_\i$} to (\r+.02, \y/4);
      }

      \foreach[count=\i from 2] \right/\span/\y in {5/4.5/3} 
      \draw[-{|[left]}, thick] (\right-\span-.02, \y/4) node[left] {$x_\i$} to (\right+.02, \y/4);
      \foreach[count=\i from 3] \left/\span/\y in {6/0.5/2}
        \draw[{|[right]}-, thick] (\left-.02, \y/4) to (\left+\span+.02, \y/4) node[right] {$x_\i$};
  
      \foreach \l/\r/\y in {3.5/6.5/4}{
        \draw[dashed] (\l-.02, \y/4) to (\r+.02, \y/4);
      }

      \draw[gray] (.5, 0) -- (8.5, 0);
      \foreach \x in {1, ..., 8}
      \draw[dashed, gray] (\x, 0) -- ++(0, .1);
    \end{tikzpicture}
    \caption{}
    \label{fig:prime-Q-node-f}
  \end{subfigure}  
  
  \begin{subfigure}[b]{0.3\linewidth}
    \centering
    \begin{tikzpicture}[scale=.75]
      \node[a-vertex] (v) at (0, 1.) {};
      \coordinate (lv) at (0,1.3) {};
      \node[b-vertex] (u) at (0, 0) {};
      \coordinate (lu) at (0,-0.3) {};
      \node[empty vertex] (x1) at (-2, 0) {};
      \coordinate (lx1) at (-2,-0.3) {};
      \node[empty vertex] (x2) at (-1, 0) {};
      \coordinate (lx2) at (-1,-0.3) {};
      \node[empty vertex] (x3) at (1, 0) {};
      \coordinate (lx3) at (1,-0.3) {};
      \node[empty vertex] (x4) at (2, 0) {};
      \coordinate (lx4) at (2,-0.3) {};
      
      \draw[ultra thick, teal] (v) -- (u) {};
      \draw (v) -- (x1) {};
      \draw (v) -- (x2) {};
      \draw (v) -- (x3) {};
      \draw (v) -- (x4) {};

      \draw (x1) -- (x2)--(u)--(x3)--(x4) {};
        \tikzstyle{every node}=[inner sep=1pt]
        \begin{footnotesize}
        \node at (lv) {$v$};
        \node at (lu) {$u$};
        \node at (lx1) {$x_1$};
        \node at (lx2) {$x_2$};
        \node at (lx3) {$x_3$};
        \node at (lx4) {$x_4$};
        \end{footnotesize}
     \end{tikzpicture}
    \caption{}
    \label{fig:fc-prime-Q-node-d}
  \end{subfigure}
  \,
      \begin{subfigure}[b]{0.17\linewidth}
    \centering
    \begin{tikzpicture}[scale=.75]
      \node[empty vertex] (x1) at (0,0) {};
      \coordinate (lx1) at (0,-0.3) {};
      \node[a-vertex] (x0) at (1, 0) {};
      \coordinate (lx0) at (1,-0.3) {};
      \node[a-vertex] (v) at (2, 0) {};
      \coordinate (lv) at (2,-0.3) {};
      \node[empty vertex] (x4) at (3,0) {};
      \coordinate (lx4) at (3,-0.3) {};
      \node[b-vertex] (u) at (1, 1) {};
      \coordinate (lu) at (1,1.3) {};
      \node[empty vertex] (x2) at (2, 1) {};
      \coordinate (lx2) at (2,1.3) {};
      \draw (x1)--(x0)--(v)--(x4);
      \draw[ultra thick, teal] (x0)--(u);
      \draw[ultra thick, teal] (v)--(u);
      \draw (x0)--(x2);
      \draw (v)--(x2);
        \tikzstyle{every node}=[inner sep=1pt]
        \begin{footnotesize}
        \node at (lv) {$v$};
        \node at (lu) {$u$};
        \node at (lx0) {$x_0$};
        \node at (lx1) {$x_1$};
        \node at (lx2) {$x_2$};
        \node at (lx4) {$x_4$};
        \end{footnotesize}
    \end{tikzpicture}
    \caption{}
    \label{fig:fc-prime-Q-node-e1}
  \end{subfigure}
    \begin{subfigure}[b]{0.17\linewidth}
    \centering
    \begin{tikzpicture}[scale=.75]
      \node[empty vertex] (x2) at (1.0,0) {};
      \coordinate (lx2) at (1,-0.3) {};
      \node[empty vertex] (x3) at (4.0,0) {};
      \coordinate (lx3) at (4,-0.3) {};
      \node[b-vertex] (x0) at (2, 0) {};
      \coordinate (lx0) at (2,-0.3) {};
      \node[b-vertex] (u) at (3, 0) {};
      \coordinate (lu) at (3,-0.3) {};
      \node[a-vertex] (v) at (2.5, 1) {};
      \coordinate (lv) at (2.8,1) {};
      \node[empty vertex] (x4) at (2.5,1.75) {};
      \coordinate (lx4) at (2.8,1.75) {};
      \draw (v) -- (x4)  {};
      \draw[ultra thick, teal] (v) -- (x0);
      \draw[ultra thick, teal] (v) -- (u);
      \draw (v) -- (x2) {};
      \draw (v) -- (x3) {};
      \draw (x2) -- (x0) -- (u) --(x3) {};
        \tikzstyle{every node}=[inner sep=1pt]
        \begin{footnotesize}
        \node at (lv) {$v$};
        \node at (lu) {$u$};
        \node at (lx0) {$x_0$};
        \node at (lx2) {$x_2$};
        \node at (lx4) {$x_4$};
        \node at (lx3) {$x_3$};
        \end{footnotesize}
    \end{tikzpicture}
    \caption{}
    \label{fig:fc-prime-Q-node-e2}
  \end{subfigure}
  \,
  \begin{subfigure}[b]{0.3\linewidth}
  \centering
    \begin{tikzpicture}[xscale=.75, yscale=0.75]
      \node[b-vertex] (u) at (1.2, 0) {};
      \coordinate (lu) at (1.2,-0.3) {};
      \node[a-vertex] (v) at (-0.6, 0.875) {};
      \coordinate (lv) at (-0.95,0.875) {};
      \node[empty vertex] (x1) at (0, 1.75) {};
      \coordinate (lx1) at (0.3,1.75) {};
      \node[empty vertex] (x2) at (0.6, .875) {};
      \coordinate (lx2) at (0.9,0.875) {};
      \node[empty vertex] (x4) at (0, 0) {};
      \node[empty vertex] (x3) at (-1.2, 0) {};
      \coordinate (lx3) at (-1.2, -0.3) {};
      \draw (x3)--(x4);
      \draw[dashed] (x4) edge (u);
      \draw (x2) -- (u);
      \draw (x1) -- (x2);
      \draw (x2) -- (x4);
      \draw (v) -- (x1);
      \draw (v) -- (x2);
      \draw (v) -- (x4);
      \draw (v) -- (x3);
      \draw [ultra thick, teal] (v) -- (u);
      
    \tikzstyle{every node}=[inner sep=1pt]
    
    \begin{footnotesize}
    \node at (lu) {$u$};
    \node at (lv) {$v$};
    \node at (lx1) {$x_1$};
    \node at (lx2) {$x_2$};
    \node at (lx3) {$x_3$};
    \end{footnotesize}
     \end{tikzpicture}    
    \caption{}
    \label{fig:fc-prime-Q-node-f}
  \end{subfigure}

  \caption{Forbidden configurations encountered when Claim~\eqref{eq:uv_pair} is not satisfied.}
  \label{fig:fs-prime-Q-node}
\end{figure}
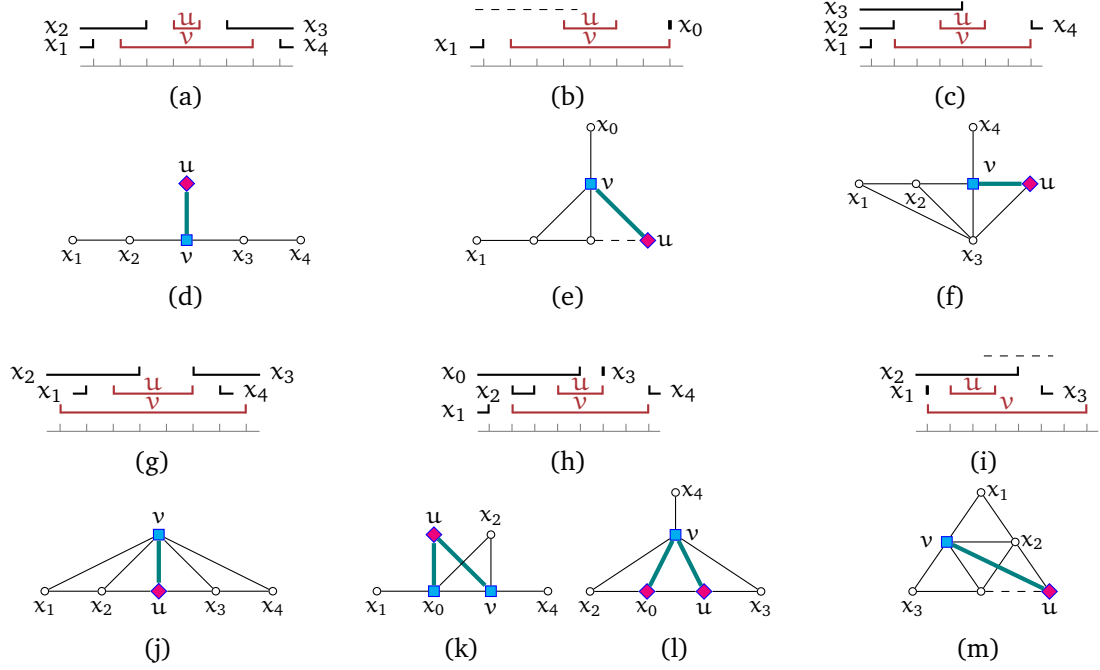

  Case 1, $L\cup R$ is disjoint from~$N(u)$.
If neither $L$ nor~$R$ is empty (see Figure~\ref{fig:prime-Q-node-a}), then $G^K$ contains Configuration~\ref{fig:p5+1-unlabeled} (reproduced and labeled in Figure~\ref{fig:fc-prime-Q-node-a}), where
    $x_{1}\in V_{\lc(v) - 1}\setminus V_{\lc(v)}$, $x_{2}\in L$,
  and
    $x_{3}\in R$, $x_{4}\in V_{\rc(v) + 1}\setminus V_{\rc(v)}$.
    Hence, we assume without loss of generality that $L \ne\emptyset$ and~$R=\emptyset$.
    We argue that $G^K$ contains Configuration~\ref{fig:dag+e-unlabeled} (reproduced and labeled in Figure~\ref{fig:fc-prime-Q-node-b}).
   We take a vertex~$x_{0}\in V_{\rc(v)}\setminus V_{\rc(v) - 1}$ and a vertex~$x_{1}\in V_{\lc(v) - 1}\setminus V_{\lc(v)}$; see Figure~\ref{fig:prime-Q-node-b}.
   By Proposition~\ref{prop:V-sets}, $V_{i-1}\cap V_{i}\setminus V_{\rc(v)}\ne \emptyset$ for all~$i = \lc(v), \ldots, \lc(u)$.
   Thus, $x_{1}$ and~$u$ are in the same component of~$G^K[A] - V_{\rc(v)}$.
   We find a shortest \stpath{x_{1}}{u} in~$G^K[A] - V_{\rc(v)}$.
   Note that the length of this path is at least two (because $L\cap N(u) = \emptyset$), 
   and its internal vertices are adjacent to~$v$ but not $x_{0}$.

  Case 2, $u$ has at least one neighbour in~$L\cup R$.
  We may assume without loss of generality that $L\cap N(u)\ne\emptyset$; it is symmetric if $R\cap N(u)\ne\emptyset$.
  If $L \not\subseteq N(u)$ (see Figure \ref{fig:prime-Q-node-c}), then $G^K$ contains Configuration~\ref{fig:whipping-top-1-unlabeled} (reproduced and labeled in Figure~\ref{fig:fc-prime-Q-node-c}), where $x_{1}\in V_{\lc(v) - 1}\setminus V_{\lc(v)}$, $x_{2}\in L \setminus N(u)$, $x_{3}\in L\cap N(u)$, and~$x_{4}\in V_{\rc(v)}\setminus V_{\rc(v) - 1}$.
Henceforth, $L$ is a nonempty subset of~$N(u)$.

  \begin{itemize}
  \item 
  Case 2.1, $L\not\subseteq V_{\rc(u)}$.  If $V_{\rc(u)}\cap V_{\rc(u)+1}\not\subseteq V_{\rc(u) - 1}$ (see Figure~\ref{fig:prime-Q-node-d}),
  then $G^K$ contains Configuration~\ref{fig:p5x1-unlabeled} (reproduced and labeled in Figure~\ref{fig:fc-prime-Q-node-d}), where $x_{1}\in V_{\lc(u)-1}\setminus V_{\lc(u)}$, $x_{2}\in L\setminus V_{\rc(u)}$, $x_{3}\in V_{\rc(u)}\cap V_{\rc(u)+1}\setminus V_{\rc(u) - 1}$, and~$x_{4}\in V_{\rc(u)+1}\setminus V_{\rc(u)}$.
Otherwise, we find a vertex~$x_{1}\in V_{\lc(v) - 1}\setminus V_{\lc(v)}$, a vertex~$x_{2}\in V_{\lc(u) - 1}\setminus V_{\lc(u)}$, a vertex~$x_{3}\in V_{\rc(u)}\setminus V_{\rc(u) - 1}$, and a vertex~$x_{4}\in V_{\rc(v)}\setminus V_{\rc(v) - 1}$; see Figure~\ref{fig:clique-path-5}.
    Since $L\subseteq N(u)$, the vertex~$x_{2}$ cannot be in~$V_{\lc(v) - 1}$; since
    $V_{\rc(u)}\cap V_{\rc(u)+1}\setminus V_{\rc(u) - 1}= \emptyset$,
    the vertex~$x_{3}$ cannot be in~$V_{\rc(u) + 1}$.
Note that $x_0u \in E(G)$ if $x_0 \in K$  (otherwise we have $N_{G^K}(u) \subseteq N_{G^K}(x_0)$, which is not the case) 
    and $vx_0 \in E(G)$ if $x_0 \notin K$ (otherwise we have $N_{G^K}(x_0) \subseteq N_{G^K}(v)$, which is not the case).
    Hence, $G^K[A]$ contains Configuration~\ref{fig:ab-wheel-unlabeled} (reproduced and labeled in Figure~\ref{fig:fc-prime-Q-node-e1})
    if $x_{0}\in K$, or Figure~\ref{fig:(p4+p1)*1-unlabeled} (reproduced and labeled in Figure \ref{fig:fc-prime-Q-node-e2}) otherwise, 
    with vertex sets~$\{u, v, x_{0}, x_{1}, x_{2}, x_{4}\}$ and~$\{u, v, x_{0}, x_{2}, x_{3}, x_{4}\}$, respectively.

  \item Case 2.2, $L\subseteq V_{\rc(u)}$.
    We argue that $G^K$ contains Configuration~\ref{fig:ddag+2e-unlabeled} (reproduced and labeled in Figure~\ref{fig:fc-prime-Q-node-f}).
    We take a vertex~$x_{1}\in V_{\lc(v)}\setminus V_{\lc(v) + 1}$, a vertex~$x_{2}$ from~$L$ such that $\rc(x_{2})$ is minimized, and a vertex~$x_{3}\in V_{\rc(x_{2}) + 1}\setminus V_{\rc(x_{2})}$. 
    See Figure~\ref{fig:prime-Q-node-f}.
By Proposition~\ref{prop:V-sets}, $V_{i}\cap V_{i+1}\setminus V_{\lc(v)}\ne \emptyset$ for all~$i = \rc(u), \rc(u)+1, \ldots, \rc(x_{2})$.
   Thus, $u$ and~$x_{3}$ are in the same component of~$G^K[A] - V_{\lc(v)}$.
   We find a shortest \stpath{u}{x_{3}} in~$G^K[A] - V_{\lc(v)}$.
   Note that all internal vertices of this path, which is nontrivial because $x_{3}u\not\in E(G^K)$, are adjacent to~$x_{2}$ but not~$x_{1}$.\qedhere
  \end{itemize}
\end{proof}

  \begin{figure}[h]
  \centering
  \begin{subfigure}[t]{0.5\linewidth}
    \centering
    \begin{tikzpicture}[xscale=1.25,yscale=0.7,>=latex,shorten >=-0.4pt,shorten <=-0.4pt]
      
      \draw[region](.8, 0.5) rectangle (2.2, 2);
      \foreach \i in {1, 2} {
        \draw[region](2.5*\i - 2.5, .5) rectangle (2.5 * \i - 2., 2);
        \draw[{|[right]}-{|[left]}] (2.5*\i - 2.4, 1.) -- node[midway, above] {$x_{\i}$}(2.5*\i - 2.2, 1.);
      }
      \draw[{|[right]}-{|[left]}] (.8, 1.3) -- node[midway, above]{$u_{1}$} (1.3, 1.3);
      \draw[{|[right]}-{|[left]}] (.8, 1.) -- node[midway, below]{$u_{2}$} (1.5, 1.);
      \draw[{|[right]}-{|[left]}] (1.8, 1.3) -- node[midway, above]{$x_{0}$} (2.2, 1.3);
      \node at (1, 0.2) {$B_{t}$};
      
      \draw[{|[right]}-] (0.8, 3. - 1/3) node[left] {$v_{1}$} -- (3.2, 3. - 1/3);
      \draw[-{|[left]}] (0, 3. - 2/3) -- (2.2, 3. - 2/3) node[right] {$v_{2}$} ;
    \end{tikzpicture}
    \caption{}
    \label{fig:Q-node-case-B-u1u2}
  \end{subfigure}

  \begin{subfigure}[b]{0.22\linewidth}
    \centering
    \begin{tikzpicture}[scale=.75]
      \node[empty vertex] (x1) at (0,0) {};
      \coordinate (lx1) at (0,-0.3) {};
      \node[a-vertex] (x0) at (1, 0) {};
      \coordinate (lx0) at (1,-0.3) {};
      \node[a-vertex] (v) at (2, 0) {};
      \coordinate (lv) at (2,-0.3) {};
      \node[empty vertex] (x4) at (3,0) {};
      \coordinate (lx4) at (3,-0.3) {};
      \node[b-vertex] (u) at (1, 1) {};
      \coordinate (lu) at (1,1.3) {};
      \node[empty vertex] (x2) at (2, 1) {};
      \coordinate (lx2) at (2,1.3) {};
      \draw (x1)--(x0)--(v)--(x4);
      \draw[ultra thick, teal] (x0)--(u);
      \draw[ultra thick, teal] (v)--(u);
      \draw (x0)--(x2);
      \draw (v)--(x2);
      \tikzstyle{every node}=[inner sep=1pt]
      \begin{footnotesize}
        \node at (lv) {$v_2$};
        \node at (lu) {$u_1$};
        \node at (lx0) {$v_1$};
        \node at (lx1) {$x_1$};
        \node at (lx2) {$x_0$};
        \node at (lx4) {$x_2$};
      \end{footnotesize}
    \end{tikzpicture}
    \caption{}
    \label{fig:fs-Q-node-case-B-u}
  \end{subfigure}
  \hspace{1cm}
  \begin{subfigure}[b]{0.22\linewidth}
    \centering
    \begin{tikzpicture}[scale=.85]
      \node[empty vertex] (x1) at (0,0) {};
      \coordinate (lx1) at (0,-0.3) {};
      \node[a-vertex] (v1) at (1, 0) {};
      \coordinate (lv1) at (1,-0.3) {};
      \node[a-vertex] (v2) at (2, 0) {};
      \coordinate (lv2) at (2,-0.3) {};
      \node[empty vertex] (x2) at (3,0) {};
      \coordinate (lx2) at (3,-0.3) {};
      \node[b-vertex] (u1) at (0.5, 1) {};
      \coordinate (lu1) at (0.5,1.3) {};
      \node[b-vertex] (u2) at (1.5, 1) {};
      \coordinate (lu2) at (1.5,1.3) {};
      \node[empty vertex] (x0) at (2.5, 1) {};
      \coordinate (lx0) at (2.5,1.3) {};
      \draw (x1)--(v1)--(v2)--(x2);
      \draw[ultra thick, teal] (v1)--(u1);
      \draw[ultra thick, teal] (v2)--(u2);
      \draw (v1)--(u2);
      \draw (v2)--(u1);
      \draw (u1)--(u2);
      \draw (v1)--(x0);
      \draw (v2)--(x0);
      \tikzstyle{every node}=[inner sep=1pt]
      \begin{footnotesize}
        \node at (lv1) {$v_1$};
        \node at (lv2) {$v_2$};
        \node at (lu1) {$u_1$};
        \node at (lu2) {$u_2$};
        \node at (lx1) {$x_1$};
        \node at (lx2) {$x_2$};
        \node at (lx0) {$x_0$};
      \end{footnotesize}
\end{tikzpicture}
    \caption{}
    \label{fig:fs-Q-node-case-B-u1u2}
  \end{subfigure}

  \caption{Forbidden configurations encountered when Claim~\eqref{eq:B} is violated.}
  \label{fig:Q-node-cases} 
\end{figure}

\begin{proof}[Proof of Claim~\eqref{eq:B}]
  \renewcommand\qedsymbol{$\lrcorner$}
  Suppose for contradiction that $u_{1}, u_{2}\in K_{1}\setminus K_{s_t}$, where~$u_{1}$ and~$u_{2}$ might be the same (see Figure~\ref{fig:Q-node-case-B-u1u2}).
  The other case is handled similarly.
Note that~$1 < t < \ell$.
  We take a vertex~$x_{0} \in K_{s_t}\setminus K_{s_t - 1}$ from the other end.
  Note that $x_0u_i \notin E(G^K)$ for $i=1,2$.
  We take a vertex~$x_1 \in V_{t-1} \setminus V_t$ and a vertex~$x_2 \in V_{t+1} \setminus V_t$.
If~$u_1 = u_2$, then the subgraph~$G^{K}[\{x_1,v_1,v_2,x_2,u_1,x_0\}]$ is Configuration~\ref{fig:ab-wheel-unlabeled} (reproduced and labeled in Figure~\ref{fig:fs-Q-node-case-B-u}).
  We are in the previous case if~$(u_{1}, v_{2})$ or~$(u_{2}, v_{1})$ is in~$\mathcal{P}$.
  Otherwise, the subgraph~$G^{K}[\{x_1,v_1,v_2,x_2,u_1,u_2,x_0\}]$ is Configuration~\ref{fig:add-1-unlabeled} (reproduced and labeled in Figure~\ref{fig:fs-Q-node-case-B-u1u2}).
\end{proof}

\begin{figure}[h]
  \centering
  \begin{subfigure}[t]{0.47\linewidth}
    \centering
    \begin{tikzpicture}[xscale=1.25,yscale=0.7,>=latex,shorten >=-0.4pt,shorten <=-0.4pt, label distance=-4pt]
      \draw[region](0, 0.5) rectangle (2, 2);
      \draw[region](2.3, 0.5) rectangle (3, 2);

      \foreach \i/\p in {1/above, 2/below} {
\pgfmathsetmacro{\y}{1.2 + Mod(\i, 2)/5}        
        \draw[{|[right]}-{|[left]}] (\i*1.5 - 1.5, \y) -- node[midway, "$x_{\i}$" \p] {} ({\i*1.5- 1}, \y);
\draw[{|[left]}-{|[right]}] (2.7 - \i*.7, \y) -- node[midway, "$u_{\the\numexpr3-\i\relax}$" \p] {} (1.4 - \i*.7, \y);
}
      \draw[{|[right]}-{|[left]}] (2.5, 1.3) --node[midway, above] {$x_{0}$} (2.8, 1.3);
      \node at (1, 0.2) {$B_{t}$};
      \node at (3, 0.2) {$B_{t+1}$};
      
      \draw[{|[right]}-{|[left]}] (0, 3. - 1/3) node[left] {$v_{1}$} -- (3, 3. - 1/3);
      \draw[{|[right]}-] (0, 3. - 2/3) node[left] {$v_{2}$} -- (3, 3. - 2/3);
      \draw[dotted] (3, 3. - 2/3) -- ++(.3, 0);
\end{tikzpicture}
    \caption{}
    \label{fig:Q-node-case-A-d}
  \end{subfigure}
  \,
  \begin{subfigure}[t]{0.47\linewidth}
    \centering
    \begin{tikzpicture}[xscale=1.25,yscale=0.7,>=latex,shorten >=-0.4pt,shorten <=-0.4pt]
      \draw[region](0, 0.5) rectangle (1.5, 2);
\foreach \i in {1, 2} {
        \draw[region](.8*\i+1.4, 0.5) rectangle (.8*\i + 1.9, 2);
        \draw[{|[right]}-{|[left]}] (\i*1. - 1., 1.3) node[above right] {$u_{\i}$} -- ({\i*1.- .5}, 1.3);
      }
      \draw[{|[right]}-{|[left]}] (3.1, 1.3) node[above right] {$x_{0}$} -- (3.4, 1.3);
      \node at (1, 0.2) {$B_{t}$};
      \node at (3.2, 0.2) {$B_{r}$};
      
      \draw[{|[right]}-{|[left]}] (0, 3.) node[left] {$v_{1}$} -- (3.5, 3.);
      \draw[{|[right]}-] (0, 3. - 1/3) node[left] {$v_{2}$} -- (3.5, 3. - 1/3);
      \draw[dotted] (3.5, 3. - 1/3) -- ++(.3, 0);
      \draw[-{|[left]}] (0, 3. - 2/3) -- (2.7, 3. - 2/3) node[right] {$x_{1}$};
    \end{tikzpicture}
    \caption{}
    \label{fig:Q-node-case-A-a}
  \end{subfigure}

  \begin{subfigure}[b]{0.22\linewidth}
    \centering
    \begin{tikzpicture}[scale=.85]
      \node[empty vertex] (x1) at (1.0,0) {};
      \coordinate (lx1) at (1,-0.4) {};
      \node[empty vertex] (x2) at (4.0,0) {};
      \coordinate (lx2) at (4,-0.4) {};
      \node[b-vertex] (u1) at (2, 0) {};
      \coordinate (lu1) at (2,-0.4) {};
      \node[b-vertex] (u2) at (3, 0) {};
      \coordinate (lu2) at (3,-0.4) {};
      \node[a-vertex] (v) at (2.5, 1) {};
      \coordinate (lv) at (2.9,1) {};
      \node[empty vertex] (x0) at (2.5,1.75) {};
      \coordinate (lx0) at (2.9,1.75) {};
      \draw (v) -- (x0)  {};
      \draw[ultra thick, teal] (v) -- (u1);
      \draw[ultra thick, teal] (v) -- (u2);
      \draw (v) -- (x1) {};
      \draw (v) -- (x2) {};
      \draw (x1) -- (u1) -- (u2) --(x2) {};
      \tikzstyle{every node}=[inner sep=1pt]
      \begin{footnotesize}
        \node at (lv) {$v_1$};
        \node at (lu1) {$u_1$};
        \node at (lu2) {$u_2$};
        \node at (lx0) {$x_0$};
        \node at (lx1) {$x_1$};
        \node at (lx2) {$x_2$};
      \end{footnotesize}
    \end{tikzpicture}
    \caption{}
    \label{fig:fs-Q-node-case-B-c}
  \end{subfigure}
  \begin{subfigure}[b]{0.22\linewidth}
    \centering
    \begin{tikzpicture}[scale=.85]
      \node[empty vertex] (x1) at (1.0,0) {};
      \coordinate (lx1) at (1,-0.4) {};
      \node[empty vertex] (x2) at (4.0,0) {};
      \coordinate (lx2) at (4,-0.4) {};
      \node[b-vertex] (u1) at (2, 0) {};
      \coordinate (lu1) at (2,-0.4) {};
      \node[b-vertex] (u2) at (3, 0) {};
      \coordinate (lu2) at (3,-0.4) {};
      \node[a-vertex] (v1) at (2, 1) {};
      \coordinate (lv1) at (1.6,1) {};
      \node[a-vertex] (v2) at (3, 1) {};
      \coordinate (lv2) at (3.4,1) {};
      \node[empty vertex] (x0) at (2.5,1.75) {};
      \coordinate (lx0) at (2.9,1.75) {};
      \draw (v1) -- (x0)  {};
      \draw (v1) -- (v2)  {};
      \draw (v2) -- (x0)  {};
      \draw[ultra thick, teal] (v1) -- (u1);
      \draw[ultra thick, teal] (v2) -- (u2);
      \draw (v1) -- (x1) {};
      \draw (v1) -- (x2) {};
      \draw (v1) -- (u2) {};
      \draw (v2) -- (u1) {};
      \draw (v2) -- (x1) {};
      \draw (v2) -- (x2) {};
      \draw (x1) -- (u1) -- (u2) --(x2) {};
      \tikzstyle{every node}=[inner sep=1pt]
      \begin{footnotesize}
        \node at (lv1) {$v_1$};
        \node at (lv2) {$v_2$};
        \node at (lu1) {$u_1$};
        \node at (lu2) {$u_2$};
        \node at (lx0) {$x_0$};
        \node at (lx1) {$x_1$};
        \node at (lx2) {$x_2$};
      \end{footnotesize}
    \end{tikzpicture}
    \caption{}
    \label{fig:fs-Q-node-case-A-d}
  \end{subfigure}
  \quad
  \quad
  \begin{subfigure}[b]{0.2\linewidth}
    \centering
    \begin{tikzpicture}[scale=0.85]
      \node[empty vertex] (x1) at (0,0) {};
      \coordinate (lx1) at (0,-0.3) {};
      \node[b-vertex] (u1) at (-1.2, 0) {};
      \coordinate (lu1) at (-1.2,-0.3) {};
      \node[b-vertex] (u2) at (1.2, 0) {};
      \coordinate (lu2) at (1.2,-0.3) {};
      \node[a-vertex] (v) at (0, 1) {};
      \coordinate (lv) at (0.3,1) {};
      \node[empty vertex] (x0) at (0,2) {};
      \coordinate (lx0) at (0.3,2) {};
      \draw (v) -- (x0)  {};
      \draw[ultra thick, teal] (v) -- (u1);
      \draw[ultra thick, teal] (v) -- (u2);
      \draw (v) -- (x1) {};
      \draw (u1) -- (x1) -- (u2) {};
      \tikzstyle{every node}=[inner sep=1pt]
      \begin{footnotesize}
        \node at (lv) {$v_1$};
        \node at (lu1) {$u_1$};
        \node at (lu2) {$u_2$};
        \node at (lx1) {$x_1$};
        \node at (lx0) {$x_0$};
      \end{footnotesize}
    \end{tikzpicture}
    \caption{}
    \label{fig:fs-Q-node-case-B-a}
  \end{subfigure}
  \begin{subfigure}[b]{0.2\linewidth}
    \centering
    \begin{tikzpicture}[scale=0.85]
      \node[empty vertex] (x1) at (0,0) {};
      \coordinate (lx1) at (0,-0.3) {};
      \node[b-vertex] (u1) at (-1.2, 0) {};
      \coordinate (lu1) at (-1.2,-0.3) {};
      \node[b-vertex] (u2) at (1.2, 0) {};
      \coordinate (lu2) at (1.2,-0.3) {};
      \node[a-vertex] (v1) at (-0.6, 1) {};
      \coordinate (lv1) at (-0.9,1) {};
      \node[a-vertex] (v2) at (0.6, 1) {};
      \coordinate (lv2) at (0.9,1) {};
      \node[empty vertex] (x0) at (0,2) {};
      \coordinate (lx0) at (0.3,2) {};
      \draw (v1) -- (x0)  {};
      \draw (v2) -- (x0)  {};
      \draw (v1) -- (v2)  {};
      \draw[ultra thick, teal] (v1) -- (u1);
      \draw[ultra thick, teal] (v2) -- (u2);
      \draw (v1) -- (x1) {};
      \draw (v1) -- (u2) {};
      \draw (v2) -- (u1) {};
      \draw (v2) -- (x1) {};
      \draw (u1) -- (x1) -- (u2) {};
      \tikzstyle{every node}=[inner sep=1pt]
      \begin{footnotesize}
        \node at (lv1) {$v_1$};
        \node at (lv2) {$v_2$};
        \node at (lu1) {$u_1$};
        \node at (lu2) {$u_2$};
        \node at (lx1) {$x_1$};
        \node at (lx0) {$x_0$};
      \end{footnotesize}
    \end{tikzpicture}
    \caption{}
    \label{fig:fs-Q-node-case-A-a}
  \end{subfigure}

  \begin{subfigure}[t]{0.8\linewidth}
    \centering
    \begin{tikzpicture}[xscale=1.25,yscale=0.7,>=latex,shorten >=-0.4pt,shorten <=-0.4pt]
      \draw[region](0, 0.5) rectangle (1.5, 2);
\foreach \i in {1, 2, 3} {
        \draw[region](.8*\i+1.4, 0.5) rectangle (.8*\i + 1.9, 2);
      }        
        \foreach \i in {1, 2} {
          \draw[{|[right]}-{|[left]}] (\i*1. - 1., 1.3) node[above right] {$u_{\i}$} -- ({\i*1.- .5}, 1.3);
      }
      \node at (3.25, 0.2) {$B_{r}$};
      \node at (4.25, 0.2) {$B_{r+1}$};
\node at (1, 0.2) {$B_{t}$};
      
      \draw[{|[right]}-{|[left]}] (0, 3.) node[left] {$v_{1}$} -- (3.5, 3.);
      \draw[{|[right]}-] (0, 3. - 1/3) node[left] {$v_{2}$} -- (3.5, 3. - 1/3);
      \draw[dotted] (3.5, 3. - 1/3) -- ++(.5, 0);
      \draw[{|[right]}-] (2.2, 3. - 2/3) node[left] {$x_{2}$} -- (4.3, 3. - 2/3);
      \draw[{|[right]}-{|[left]}] (3.9, 1.3) node[above right] {$x_{1}$} -- (4.3, 1.3);
    \end{tikzpicture}
    \caption{}
    \label{fig:Q-node-case-A-c}
  \end{subfigure}

  \begin{subfigure}[b]{0.22\linewidth}
    \centering
    \begin{tikzpicture}[scale=.85]
      \node[empty vertex] (x2) at (-2,0) {};
      \coordinate (lx2) at (-2,0.3) {};
      \node[empty vertex] (x1) at (-1,0) {};
      \coordinate (lx1) at (-1,0.3) {};
      \node[a-vertex] (v) at (0,0) {};
      \coordinate (lv) at (0,0.3) {};
      \node[b-vertex] (u1) at (1,0.5) {};
      \coordinate (lu1) at (1,0.8) {};
      \node[b-vertex] (u2) at (1,-0.5) {};
      \coordinate (lu2) at (1,-0.85) {};
      \draw[ultra thick, teal] (v) -- (u1) {};
      \draw[ultra thick, teal] (v) -- (u2) {};
      \draw (x2)--(x1);    
      \draw (x1)--(v);    
      \tikzstyle{every node}=[inner sep=1pt]
      \begin{footnotesize}
        \node at (lv) {$v_1$};
        \node at (lu1) {$u_1$};
        \node at (lu2) {$u_2$};
        \node at (lx1) {$x_2$};
        \node at (lx2) {$x_1$};
      \end{footnotesize}
\end{tikzpicture}
    \caption{}
    \label{fig:fs-Q-node-case-B-b}
  \end{subfigure}
  \begin{subfigure}[b]{0.22\linewidth}
    \centering
    \begin{tikzpicture}[scale=.85]
      \node[empty vertex] (x2) at (-2,0) {};
      \coordinate (lx2) at (-2,0.3) {};
      \node[empty vertex] (x1) at (-1,0) {};
      \coordinate (lx1) at (-1,0.3) {};
      \node[a-vertex] (v1) at (0,0.5) {};
      \coordinate (lv1) at (0,0.8) {};
      \node[a-vertex] (v2) at (0,-0.5) {};
      \coordinate (lv2) at (0,-0.85) {};
      \node[b-vertex] (u1) at (1,0.5) {};
      \coordinate (lu1) at (1,0.8) {};
      \node[b-vertex] (u2) at (1,-0.5) {};
      \coordinate (lu2) at (1,-0.85) {};
      \draw[ultra thick, teal] (v1) -- (u1) {};
      \draw[ultra thick, teal] (v2) -- (u2) {};
      \draw (x2)--(x1);    
      \draw (x1)--(v1);    
      \draw (x1)--(v2);    
      \draw (v1)--(v2);    
      \draw (v1) -- (u2) {};
      \uncertain{u1}{v2};
      \tikzstyle{every node}=[inner sep=1pt]
      \begin{footnotesize}
        \node at (lv1) {$v_1$};
        \node at (lv2) {$v_2$};
        \node at (lu1) {$u_1$};
        \node at (lu2) {$u_2$};
        \node at (lx1) {$x_2$};
        \node at (lx2) {$x_1$};
      \end{footnotesize}
\end{tikzpicture}
    \caption{}
    \label{fig:fs-Q-node-case-A-b}
  \end{subfigure}
  \,
  \begin{subfigure}[b]{0.22\linewidth}
    \centering
    \begin{tikzpicture}[scale=.85]
      \node[empty vertex] (x2) at (-2,0) {};
      \coordinate (lx2) at (-2,0.3) {};
      \node[empty vertex] (x1) at (-1,0) {};
      \coordinate (lx1) at (-1,0.3) {};
      \node[a-vertex] (v1) at (0,0.5) {};
      \coordinate (lv1) at (0,0.8) {};
      \node[a-vertex] (v2) at (0,-0.5) {};
      \coordinate (lv2) at (0,-0.85) {};
      \node[b-vertex] (u1) at (1,0.5) {};
      \coordinate (lu1) at (1,0.8) {};
      \node[b-vertex] (u2) at (1,-0.5) {};
      \coordinate (lu2) at (1,-0.85) {};
      \draw[ultra thick, teal] (v1) -- (u1) {};
      \draw[ultra thick, teal] (v2) -- (u2) {};
      \draw (x2)--(x1);    
      \draw (x1)--(v1);    
      \draw (x1)--(v2);    
      \draw (v1)--(v2);    
      \draw (v1) -- (u2) {};
      \uncertain{u1}{v2};
      \draw (v2) -- (x2) {};
      \tikzstyle{every node}=[inner sep=1pt]
      \begin{footnotesize}
        \node at (lv1) {$v_1$};
        \node at (lv2) {$v_2$};
        \node at (lu1) {$u_1$};
        \node at (lu2) {$u_2$};
        \node at (lx1) {$x_2$};
        \node at (lx2) {$x_1$};
      \end{footnotesize}
\end{tikzpicture}
    \caption{}
    \label{fig:fs-Q-node-case-A-c}
  \end{subfigure}

  \caption{Forbidden configurations encountered when Claim~\eqref{eq:A} is violated.}
  \label{fig:Q-node-case-A} 
\end{figure}
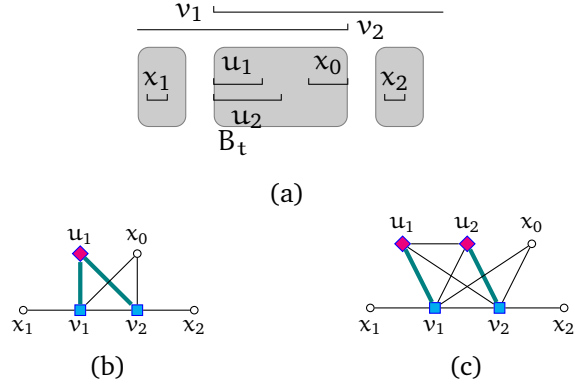

\begin{proof}[Proof of Claim~\eqref{eq:A}]
  \renewcommand\qedsymbol{$\lrcorner$}
Suppose for contradiction that $u_{1}\in K_{1}\setminus K_{s_t}$, $u_{2}\in K_{s_t}\setminus K_{1}$, and~$\lc(v_{1}) = \lc(v_{2}) = t$ or~$\rc(v_{1}) = \rc(v_{2}) = t$.
It is possible that~$v_{1} = v_{2}$.
We consider the case that~$\lc(v_{1}) = \lc(v_{2}) = t$, and the other is symmetric.

Suppose first that $u_1u_2$ is an edge of $E(G^K)$; see Figure~\ref{fig:Q-node-case-A-d}.
We take a vertex~$x_{1} \in K_{1}\setminus K_{2}$ and a vertex~$x_{2} \in K_{s_t}\setminus K_{s_t - 1}$.
Note that~$x_1u_1u_2x_2$ is an induced path in $G^K$.  We pick an arbitrary vertex~$x_{0}$ from~$\inner(B_{t+1})$.
If~$v_{1} = v_{2}$, then~$G^{K}[\{x_1,u_1,u_2,x_2,v_1,x_0\}]$
Configuration~\ref{fig:(p4+p1)*1-unlabeled} (reproduced and labeled in Figure~\ref{fig:fs-Q-node-case-B-c}).
Hence,~$v_{1} \ne v_{2}$.
We are in the previous case if~$(u_{1}, v_{2})$ or~$(u_{2}, v_{1})$ is in~$\mathcal{P}$.
Otherwise,~$G^{K}[\{x_1,u_1,u_2,x_2,v_1,v_2,x_0\}]$ is Configuration~\ref{fig:add-2-unlabeled} (reproduced and labeled in Figure~\ref{fig:fs-Q-node-case-A-d}).

In the rest, $u_1u_2$ is not an edge of $E(G^K)$.
Assume without loss of generality that $\rc(v_1) \leq \rc(v_2)$.
Let $r = \rc(v_1)$. Note that $t < r$.
\begin{itemize}
\item Case 1, $\encomp(B_t) \setminus V_r \neq \emptyset$.
  We take a vertex~$x_0 \in V_{r} \setminus V_{r-1}$ and a vertex~$x_1 \in \encomp(B_t) \setminus V_r$ (see Figure~\ref{fig:Q-node-case-A-a}).
  If~$v_{1} = v_{2}$, then~$G^{K}[\{u_1,x_1,u_2,v_1,x_0\}]$ is Configuration~\ref{fig:dag+2e-unlabeled} (reproduced and labeled in Figure~\ref{fig:fs-Q-node-case-B-a}).
Hence,~$v_{1} \ne v_{2}$.
We are in the previous case if~$(u_{1}, v_{2})$ or~$(u_{2}, v_{1})$ is in~$\mathcal{P}$.
Otherwise,~$G^{K}[\{u_1,x_1,u_2,v_1,v_2,x_0\}]$ is Configuration~\ref{fig:ddag+2e-unlabeled} (reproduced and labeled in Figure~\ref{fig:fs-Q-node-case-A-a}).

\item Case 2, $\encomp(B_t) \setminus V_r = \emptyset$.
  Note that~$v_1$ is not universal in~$G^K[A]$; otherwise $1=t$, $r=\ell$, $V_1 \cap V_2 \setminus V_{r} \neq \emptyset$ by Proposition~\ref{prop:V-sets}, and hence $\encomp(B_1) \setminus V_r \neq \emptyset$, contradicting our assumption.
  Since $v_1$ is not universal, $1 < t$ or $r < \ell$, and hence 
  $V_{t-1} \cap V_{t} \setminus V_r \neq \emptyset$ or $V_r \cap V_{r+1} \setminus V_t \neq \emptyset$ by
  Proposition~\ref{prop:V-sets}.
  Since~$V_{t-1} \cap V_{t} \subseteq \encomp(B_t)\subseteq V_r$, we have
  $V_r \cap V_{r+1} \setminus V_{t} \neq \emptyset$.
  We take a vertex~$x_{1} \in V_{r+1} \setminus V_r$ and a vertex~$x_{2} \in V_r \cap V_{r+1} \setminus V_t$ (see Figure~\ref{fig:Q-node-case-A-c}).
  If~$v_1 = v_2$, then the subgraph~$\{u_1,u_2,v_1,x_1,x_2\}$ is Configuration~\ref{fig:fork-unlabeled} (reproduced and labeled in Figure~\ref{fig:fs-Q-node-case-B-b}).
  Hence, $v_1 \neq v_2$.
  We are in the previous case if~$(u_{1}, v_{2})\in \mathcal{P}$.
Otherwise,~$G^K[\{u_1,u_2,v_1,v_2,x_1,x_2\}]$ is Configuration~\ref{fig:double-fork-unlabeled} when~$v_2 x_{1} \notin E(G^K)$ (reproduced and labeled in Figure~\ref{fig:fs-Q-node-case-A-b}), or
Configuration~\ref{fig:double-fork+1-unlabeled} when~$v_{2} x_{1} \in E(G^K)$ (reproduced and labeled in Figure \ref{fig:fs-Q-node-case-A-c}).
\qedhere
\end{itemize}
\end{proof}  

This concludes the claims and the proof of the lemma.
\end{proof}

We are now putting everything together to prove Lemma~\ref{thm:star}.
\begin{proof}[Proof of Lemma~\ref{thm:star}]
  First, we deal with necessity.  We note that any clique path of any interval graph $F$ from Figure~\ref{fig:forbidden-configurations-C_4-free} contains a pair $v \in K$ and $u \in V(G) \setminus K$ joined with a thick edge such that $\lk(v) < \lk(u) \leq \rk(u) < \rk(v)$. 
Thus, if $G^K$ contains any configuration in Figure~\ref{fig:forbidden-configurations-C_4-free}, then it does not satisfies Condition~\eqref{eq:star}.
  
  For the sufficiency, let~$T$ be a PQ-tree of~$G^{K}$, and let~$Z$ be its root.
  We use induction to show that for every node~$A$ of~$T$, the subgraph~$G^K[A]$ has an admissible clique path.
  This holds trivially for leaves, when~$G^K[A]$ is a complete graph, having only one maximal clique.
  By Lemmas~\ref{lem:p-node-induction} and~\ref{lem:q-node-induction}, we can find an admissible clique path of~$G^{K}[Z] = G^{K}$ for the root node~$Z$.  By~\eqref{eq:main_cond_inner_vertices}, $G^K$ satisfies Condition~\eqref{eq:star}.
\end{proof}

\appendix
\section*{Appendix: PQ-trees}

In this section we prove Propositions~\ref{prop:PQ-tree-nodes} and~\ref{prop:V-sets}.

\begin{proof}[Proof of Proposition~\ref{prop:PQ-tree-nodes}]
Let $u \in \inner(A)$ and $v \in \encomp(A)$.
Then $\cliques(u) \subseteq \cliques(A)$ and $\cliques(A) \subsetneq \cliques(v)$, and 
hence $\cliques(u) \subsetneq \cliques(v)$.
This proves $uv \in E(H)$ and completes the proof of~\eqref{item:PQ-tree-ex-univ}.

Let $\langle B_1, \ldots, B_{\ell}\rangle$ be an order of the children of $A$.
By~\eqref{item:PQ-tree-ex-univ} and by Theorem \ref{thm:pq-trees}, for every $v \in V(A) \setminus (\bigcup_{i=1}^\ell \inner(B_i))$ 
the set
\[
  \interval(v) = \{i \in [1,\ell] \mid \inner(B_i) \subsetneq N(v)\}
\]
is contiguous and has the size at least $2$.
Clearly, $v \in \univ(A)$ if and only if $\interval(v) = [1,\ell]$.

First, we show that $H(A) \setminus \univ(A)$ is connected when $A$ is a Q-node.
Assume for a contrary that $H[A] \setminus \univ(A)$ is not connected.
Given the above observations, it means there is $1 \leq k < \ell$ such that 
for every non-universal vertex $v \in V(A) \setminus (\bigcup_{i=1}^\ell \inner(B_i))$ in $H[A]$ we have
either $\interval(v) \subseteq [1,k]$ or $\interval(v) \subseteq [k+1, \ell]$.
Now, one can easily check that if we arrange the children of $A$ in the order
$\langle B_k, \ldots, B_1, B_\ell, \ldots, B_{k+1} \rangle$, and the other nodes of $T$ in a way described by Theorem~\ref{thm:pq-trees},
we obtain a clique order of $H$ in which the set $\cliques(v)$ for every $v \in V(H)$ is consecutive.
This, however, contradicts Theorem~\ref{thm:pq-trees}. 

Suppose $A$ is a P-node. 
We need to show that for every $v \in V(A) \setminus (\bigcup_{i=1}^\ell \inner(B_i))$ we have 
$\interval(v) = [1, \ell]$.
This holds when $v \in \encomp(A)$, so let $v \in \inner(A)$.
Anyway, if $\interval[v] \subsetneq [1, \ell]$, then one may note
some permutation of the children of $A$ leads to a clique order in which the set $\cliques(v)$ is not consecutive.
This violates Theorem~\ref{thm:pq-trees} and completes the proof of~\eqref{item:PQ-tree-connectivity}.
\end{proof}

\begin{proof}[Proof of Proposition \ref{prop:V-sets}]

Let $i < j$ be such that $1 < i < j \leq \ell$ or $1 \leq i < j < \ell$.
Assume that $V_{i-1} \cap V_{i} \setminus V_j = \emptyset$ and $V_j \cap V_{j+1} \setminus V_{i} = \emptyset$.
Then, we arrange the children of $A$ in the order 
$\langle B_1, \ldots, B_{i-1}, B_j, \ldots, B_i, B_{j+1}, \ldots B_\ell \rangle$ 
and the other nodes of $T$ in the way prescribed by Theorem~\ref{thm:pq-trees},
and we obtain a clique order in which the set $\cliques(v)$ for every $v \in V(H)$ is consecutive.
This contradicts Theorem~\ref{thm:pq-trees}.
\end{proof}

\bibliographystyle{plainurl}
\bibliography{../references}

\begin{thebibliography}{10}

\bibitem{BoothLueker76}
Kellogg~S. Booth and George~S. Lueker.
\newblock Testing for the consecutive ones property, interval graphs, and graph
  planarity using {PQ}-tree algorithms.
\newblock {\em Journal of Computer and System Sciences}, 13(3):335--379, 1976.
\newblock \href {https://doi.org/10.1016/S0022-0000(76)80045-1}
  {\path{doi:10.1016/S0022-0000(76)80045-1}}.

\bibitem{cao-24-split-cag}
Yixin Cao, Jan Derbisz, and Tomasz Krawczyk.
\newblock Characterization of circular-arc graphs: {I.} split graphs.
\newblock arXiv:2403.01947, 2024.
\newblock \href {https://doi.org/10.48550/arXiv.2403.01947}
  {\path{doi:10.48550/arXiv.2403.01947}}.

\bibitem{cao-24-cag-iii-chordal}
Yixin Cao and Tomasz Krawczyk.
\newblock Characterization of circular-arc graphs: {III.} chordal graphs.
\newblock arXiv:2409.02733, 2024.
\newblock \href {https://doi.org/10.48550/arXiv.2409.02733}
  {\path{doi:10.48550/arXiv.2409.02733}}.

\bibitem{cao-24-cag-iv-c4-free}
Yixin Cao and Tomasz Krawczyk.
\newblock Characterization of circular-arc graphs: {IV.} {$C_{4}$}-free graphs.
\newblock Manuscript, 2024.

\bibitem{fulkerson-65-interval-graphs}
Delbert~R. Fulkerson and Oliver~A. Gross.
\newblock Incidence matrices and interval graphs.
\newblock {\em Pacific Journal of Mathematics}, 15(3):835--855, 1965.
\newblock \href {https://doi.org/10.2140/pjm.1965.15.835}
  {\path{doi:10.2140/pjm.1965.15.835}}.

\bibitem{GareyJMP80}
Michael~R. Garey, David~S. Johnson, Gerald~L. Miller, and Christos~H.
  Papadimitriou.
\newblock The complexity of coloring circular arcs and chords.
\newblock {\em {SIAM} J. Algebraic Discret. Methods}, 1(2):216--227, 1980.
\newblock \href {https://doi.org/10.1137/0601025} {\path{doi:10.1137/0601025}}.

\bibitem{gimbel-88-end-vertices}
John Gimbel.
\newblock End vertices in interval graphs.
\newblock {\em Discrete Applied Mathematics}, 21(3):257--259, 1988.
\newblock \href {https://doi.org/10.1016/0166-218X(88)90071-6}
  {\path{doi:10.1016/0166-218X(88)90071-6}}.

\bibitem{hadwiger-64-combinatorial-geometry}
Hugo Hadwiger, Hans Debrunner, and Victor Klee.
\newblock {\em Combinatorial geometry in the plane}.
\newblock Athena series. Holt, Rinehart and Winston, London, 1964.

\bibitem{hsu-95-independent-set-cag}
Wen-Lian Hsu and Jeremy Spinrad.
\newblock Independent sets in circular-arc graphs.
\newblock {\em Journal of Algorithms}, 19(2):145--160, 1995.
\newblock \href {https://doi.org/10.1006/jagm.1995.1031}
  {\path{doi:10.1006/jagm.1995.1031}}.

\bibitem{joeris-11-hcag}
Benson~L. Joeris, Min~Chih Lin, Ross~M. McConnell, Jeremy~P. Spinrad, and
  Jayme~Luiz Szwarcfiter.
\newblock Linear-time recognition of {H}elly circular-arc models and graphs.
\newblock {\em Algorithmica}, 59(2):215--239, 2011.
\newblock \href {https://doi.org/10.1007/s00453-009-9304-5}
  {\path{doi:10.1007/s00453-009-9304-5}}.

\bibitem{klee-69-cag}
Victor Klee.
\newblock What are the intersection graphs of arcs in a circle?
\newblock {\em American Mathematical Monthly}, 76(7):810--813, 1969.
\newblock \href {https://doi.org/10.1080/00029890.1969.12000337}
  {\path{doi:10.1080/00029890.1969.12000337}}.

\bibitem{lekkerkerker-62-interval-graphs}
Cornelis~G. Lekkerkerker and J.~Ch. Boland.
\newblock Representation of a finite graph by a set of intervals on the real
  line.
\newblock {\em Fundamenta Mathematicae}, 51:45--64, 1962.
\newblock \href {https://doi.org/10.4064/fm-51-1-45-64}
  {\path{doi:10.4064/fm-51-1-45-64}}.

\bibitem{mcconnell-03-recognition-cag}
Ross~M. McConnell.
\newblock Linear-time recognition of circular-arc graphs.
\newblock {\em Algorithmica}, 37(2):93--147, 2003.
\newblock \href {https://doi.org/10.1007/s00453-003-1032-7}
  {\path{doi:10.1007/s00453-003-1032-7}}.

\bibitem{spinrad-88-case-1}
Jeremy~P. Spinrad.
\newblock Circular-arc graphs with clique cover number two.
\newblock {\em Journal of Combinatorial Theory, Series B}, 44(3):300--306,
  1988.
\newblock \href {https://doi.org/10.1016/0095-8956(88)90038-X}
  {\path{doi:10.1016/0095-8956(88)90038-X}}.

\end{thebibliography}

\end{document}